\documentclass[final,leqno]{siamltex}
\pdfoutput=1
\usepackage{amsmath,amssymb,amsfonts,stmaryrd,todonotes}
\usepackage{graphicx, tikz, pgfplots}
\usepgfplotslibrary{groupplots}
\usepackage{subfigure}
\usepackage{hhline}
\usepackage{algorithm,algpseudocode}
\usepackage{multirow}
\usepackage{appendix}
\usepackage{placeins}
\usepackage{tabularx}
\usepackage{mathtools}
\mathtoolsset{showonlyrefs=true}
\usepackage{array}

   \newcommand{\h}{{h}}

   \makeatletter
   \def\@tvsp{\mathchoice{{}\mkern-3mu}{{}\mkern-3mu}{{}\mkern-3mu}{}}
   \def\ltrivert{|\@tvsp|\@tvsp|}
   \def\rtrivert{|\@tvsp|\@tvsp|}
   \makeatother

   \makeatletter
   \def\@avgsp{\mathchoice{{}\mkern-6mu}{{}\mkern-6mu}{{}\mkern-6mu}{}}
   \def\llaverage{\{\@avgsp\{}
   \def\rraverage{\}\@avgsp\}}
   \makeatother

   \newtheorem{remark}[theorem]{Remark}
   
   \newcolumntype{L}[1]{>{\raggedright\let\newline\\\arraybackslash\hspace{0pt}}m{#1}}
   \newcolumntype{C}[1]{>{\centering\let\newline\\\arraybackslash\hspace{0pt}}m{#1}}

\title{Multigrid algorithms for $hp$-Discontinuous Galerkin discretizations of elliptic problems}

\author{Paola F. Antonietti\thanks{MOX-Laboratory for Modeling and Scientific Computing, Dipartimento di Matematica, Politecnico di Milano, Piazza Leondardo da Vinci 32, 20133 Milano, Italy (\tt paola.antonietti@polimi.it)} \and Marco Sarti\thanks{MOX-Laboratory for Modeling and Scientific Computing, Dipartimento di Matematica, Politecnico di Milano, Piazza Leondardo da Vinci 32, 20133 Milano, Italy (\tt marco.sarti@polimi.it)}\and Marco Verani\thanks{MOX-Laboratory for Modeling and Scientific Computing Dipartimento di Matematica, Politecnico di Milano, Piazza Leondardo da Vinci 32, 20133 Milano, Italy (\tt marco.verani@polimi.it)}}

\begin{document}

\maketitle

\begin{abstract}
We present W-cycle multigrid algorithms for the solution of the linear system of equations arising from a wide class of $hp$-version discontinuous Galerkin discretizations of elliptic problems. Starting from a classical framework in multigrid analysis, we define a smoothing and an approximation property, which are used to prove the uniform convergence of the W-cycle scheme with respect to the granularity of the grid and the number of levels. The dependence of the convergence rate on the polynomial approximation degree $p$ is also tracked, showing that the contraction factor of the scheme deteriorates with increasing $p$. A discussion on the effects of employing inherited or non-inherited sublevel solvers is also presented. Numerical experiments confirm the theoretical results.
\end{abstract}



\pagestyle{myheadings}
\thispagestyle{plain}
\markboth{P. F. Antonietti, M. Sarti, and M. Verani}{Multigrid algorithms for $hp$-Discontinuous Galerkin discretizations of elliptic problems}

\section{Introduction}
Discontinuous Galerkin (DG) methods have undergone a huge development in the last three decades mainly because of their flexibility in dealing with a wide range of equations within the same unified framework, in handling non-conforming grids and variable polynomial approximation orders,  and in imposing weakly boundary conditions. Therefore, the construction of effective solvers such as domain decomposition and multigrid methods has become an  active research field.  
Domain decomposition methods are based on the definition of subproblems on single subdomains, followed by a {\it coarse} correction, which ensures the scalability of the method. In the framework of domain decomposition algorithms for DG methods, in \cite{FengKarak01} a Schwarz preconditioner based on overlapping and non-overlapping partitions of the domain is analyzed.  The case of non-overlapping Schwarz methods with {\it  inexact} local solvers is addressed in a unified framework in \cite{AntoAyu07,AntoAyu08}. This topic has been further analyzed in \cite{LassTos03,FengKarak05,BriCamPinDahm08,AntoAyu09,BriCamPinDahmMass09,DrySark10,BarkBrennSung11,AntAyuBrenSung12}. 
For substructuring-type preconditioners for DG methods, we mention \cite{DrySark06,DryGalSark07}, where  Neumann-Neumann and Balancing Domain Decomposition with Constraints (BDDC) for Nitsche-type methods are studied. A unified approach for BDDC is recently proposed in \cite{DioDarm12}, while in \cite{BrennParSu13} a preconditioner for an over-penalized DG method is studied.  All these contributions focus on the $h$-version of DG methods; only recently some attention has been devoted to the development of efficient solvers for $hp$-DG methods. The first contribution in this direction is in \cite{AntHou}, where a non-overlapping Schwarz preconditioner for the $hp$-version of DG methods is analyzed, cf. also \cite{AntGiaHou13} for the extension to domains with complicated geometrical details. In \cite{CanPavPie12, BrCampCanDah13} BDDC and multilevel preconditioners for the $hp$-version of a DG scheme are analyzed, in parallel with conforming methods.  Substructuring-type preconditioners for $hp$-Nitsche type methods have been studied recently in \cite{AntAyuBerPen12}. The issue of preconditioning hybrid DG methods is investigated in \cite{SchoberlLehrenfeld_2013}.
Here we are interested in multigrid algorithms for $hp$-version DG methods, that exploit the solution of suitable subproblems defined on 
different levels of discretization. The levels can  differ for the mesh-size~($h$-multigrid), the polynomial approximation degree~($p$-multigrid) or both~($hp$-multigrid). 
In the framework of $h$-multigrid algorithms for DG methods, in \cite{GopKan03} a uniform (with respect to the mesh size) multigrid  preconditioner is studied. In \cite{HemHoffRaa03,HemHoffRaa04} a Fourier analysis of a multigrid solver for a class of DG discretizations is performed, focusing on the performance of several relaxation methods, while in \cite{RaaHem05} the analysis concerns convection-diffusion equations in the convection-dominated regime. Other contributions can be found for low-order DG approximations: in \cite{BrenZhao} it is shown that V-cycle, F-cycle and W-cycle multigrid algorithms converge uniformly with respect to all grid levels, with further extensions to an over-penalized method in \cite{BrennOwe} and graded meshes in \cite{BrennCuiSu,BrenCuGuSu}. At the best of our knowledge, no theoretical results in the framework of $p$- and $hp$-DG methods are available, even though $p$-multigrid solvers are widely used in practical applications, cf. \cite{FidkOli,Luo,Nastase,MascaHelen,Shahbazi,BasGhid}, for example.\\

In this paper, we present W-cycle $hp$-multigrid schemes for high-order DG discretizations of a second order elliptic problem. We consider a wide class of symmetric DG schemes, and, following the framework  presented in \cite{BrenScott,BrenZhao,BrennCuiSu}, 
we prove that the W-cycle algorithms converge uniformly with respect to the granularity of the underlying mesh and the number of levels, but the contraction factor of the scheme deteriorates with increasing $p$. The key point of our analysis is suitable smoothing and approximation properties of the $hp$-multigrid method. The smoothing scheme is a Richardson iteration, and we exploit the spectral properties of the stiffness operator to obtain the desired estimates.  The approximation property is based on the error estimates for $hp$-DG methods shown in \cite{PerSchot,HouSchSul,StaWihl}. We also discuss in details the effects of employing inherited or non-inherited sublevel solvers. More precisely, we show that the W-cycle algorithm converges uniformly with respect to the number of levels if non-inherited sublevel solvers are employed ({\it i.e.}, the coarse solvers are built rediscretizing our original problem on each level), whereas convergence cannot be independent of the number of levels if inherited bilinear forms are considered ({\it i.e.}, the coarse solvers are the restriction of the stiffness matrix constructed on the finest grid).
Those findings are confirmed by numerical experiments. \\

The rest of the paper is organized as follows. In Section~\ref{sec:theory} we introduce the model problem and its DG discretization, and recall some results needed in the forthcoming analysis. 
In Section~\ref{mulgr} we introduce W-cycle schemes based on non-inherited bilinear forms. The convergence analysis is performed in Section~\ref{convergence}, and further extended to a wider class of symmetric DG symmetric schemes in Section~\ref{sec.extesion}. Multigrid algorithms based on employing inherited bilinear forms are discussed in Section~\ref{sec:inehrited}.
The theoretical estimates are then verified through numerical experiments in Section~\ref{numerical}. In Section~\ref{conclusion}
we draw some conclusions. Finally, in Appendix~\ref{appendixA} we report some technical results.

\section{Model problem and DG discretization}
\label{sec:theory}
In this section, we introduce the model problem and its discretization by $hp$-version DG methods.\par\medskip
Throughout the paper we will use standard notation for Sobolev spaces \cite{Adams}. We write $x\lesssim y$ in lieu of $x\leq C y$ for a positive constant $C$ independent of the discretization parameters. When needed, the constants will be written explicitly.\par\medskip
Let $\Omega\in \mathbb{R}^d$, $d = 2,3$, be a polygonal/polyhedral domain and $f$ a given function in $H^{s-1}(\Omega)$, $s\geq1$.  We consider the weak formulation of the Poisson problem, with homogeneous Dirichlet boundary conditions: find $u\in V= H^{s+1}(\Omega)\cap H_0^1(\Omega)$, such that
\begin{equation}
\int_\Omega \nabla u \cdot \nabla v\ dx =\int_\Omega f v\ dx \, \qquad \forall v\in V,
\label{weak}
\end{equation}
We will make the following elliptic regularity assumption on the solution to \eqref{weak}: 
\begin{equation}
\left\|u\right\|_{H^{s+1}(\Omega)}\lesssim \left\|f\right\|_{{H^{s-1}}(\Omega)}.
\label{elliptic}
\end{equation}
Such a hypothesis can be relaxed, cf. \cite{Bren99,BrennCuiSu,BrenCuGuSu} for multigrid methods that do not assume regularity on the solution.\par\medskip

We introduce a quasi-uniform partition $\mathcal{T}_K$ of $\Omega$ into shape-regular elements $T$ of diameter $h_T$, and set $\h_K = \max_{T\in\mathcal{T}_K} h_T$. We suppose that each element $T\in \mathcal{T}_K$ is an affine image of a reference element $\widehat{T}$, {\it i.e.}, $T = \mathsf{F}_T(\widehat{T})$, which is either the open unit simplex or the unit hypercube in $\mathbb{R}^d$, $d=2,3$. We denote by $\mathcal{F}_K^{I}$, resp.  $\mathcal{F}_K^{B}$, the set of interior, resp. boundary, faces (if $d=2$ ``face'' means ``edge'') of the partition $\mathcal{T}_K$ and set $\mathcal{F}_K=\mathcal{F}_K^{I} \cup \mathcal{F}_K^{B}$, with the convention that an interior face is the non-empty intersection of the closure of two neighboring elements. Given $s\geq 1$, the {\it  broken Sobolev space} $H^s(\mathcal{T}_K)$ is made of the functions that are in $H^s$ elementwise. The DG scheme provides a discontinuous approximation of the solution of \eqref{weak}, which in general belongs to a finite dimensional subspace of $H^s(\mathcal{T}_K)$ defined as
\begin{equation}
\label{VK}
V_K=\{v\in L^2(\Omega):v\circ \mathsf{F}_T \in \mathbb{M}^{p_K}(\widehat{T})\quad\forall T\in \mathcal{T}_K\},
\end{equation}
where $\mathbb{M}^{p_K}(\widehat{T})$ is either the space of polynomials of total degree less than or equal to $p_K\geq 1$ on the simplex $\widehat{T}$, or the space  of all tensor-product polynomials on $\widehat{T}$ of degree $p_K$ in each coordinate direction, if $\widehat{T}$ is the reference hypercube in $\mathbb{R}^d$.\par\medskip
For regular enough vector-valued and scalar functions $\boldsymbol{\tau}$ and $v$, respectively, we define the {\it  jumps} and {\it  weighted averages} (with $\delta\in[0,1]$) across the face $F\in\mathcal{F}_K$ as follows
\begin{alignat*}{4}
\llbracket \boldsymbol{\tau} \rrbracket &= \boldsymbol{\tau}^{+} \cdot \mathbf{n}_{T^+} + \boldsymbol{\tau}^{-} \cdot \mathbf{n}_{T^{-}},\qquad{}& \llaverage \boldsymbol{\tau} \rraverage_\delta &= \delta\boldsymbol{\tau}^{+}  + (1-\delta)\boldsymbol{\tau}^,\qquad&F\in \mathcal{F}_K^I,\\
\llbracket v \rrbracket &= v^{+} \mathbf{n}_{T^+} + v^{-} \mathbf{n}_{T^{-}}, &\llaverage v \rraverage_\delta{}& = \delta v|_{T^+}  + (1-\delta)v|_{T^{-}}, &F\in \mathcal{F}_K^I,
\end{alignat*}
with $\mathbf{n}_{T^{\pm}}$ denoting the outward normal vector to  $\partial T^{\pm}$, and $\boldsymbol{\tau}^{\pm}$ and $v^{\pm}$ are the traces of $\boldsymbol{\tau}$ and $v$ taken within the interior of $T^{\pm}$, respectively.
 In the case $\delta = 1/2$ (standard average) we will simply write $\llaverage\cdot\rraverage$. 
On a boundary face $F\in \mathcal{F}_K^B$, we set $\llbracket v \rrbracket = v\mathbf{n}_T{}$, $\llaverage \boldsymbol{\tau} \rraverage_\delta =\boldsymbol{\tau}|_T$.
We observe that the following relations hold on each $F\in\mathcal{F}_K^I$
\begin{equation*}
\begin{aligned}
& \llaverage u\rraverage_\delta = \llaverage u\rraverage+\boldsymbol{\delta}\cdot \llbracket u\rrbracket,\qquad \llaverage u\rraverage_{1-\delta} = \llaverage u\rraverage-\boldsymbol{\delta}\cdot \llbracket u\rrbracket,
\end{aligned}
\end{equation*}
with $\boldsymbol{\delta}=(\delta-1/2)\boldsymbol{n}_F$, being $\boldsymbol{n}_F$ the outward unit normal vector to the face $F$ to which $\delta$ is associated.\\

Next, we introduce the local lifting operators $r_F:[L^1(F)]^d\rightarrow[V_K]^d$ and $l_F:L^1(F)\rightarrow[V_K]^d$ 
\begin{alignat*}{3}
\int_{\Omega} r_F(\boldsymbol{\tau})\cdot\boldsymbol{\eta}\ dx &=-\int_{F} \boldsymbol{\tau}\cdot\llaverage\boldsymbol{\eta}\rraverage\ ds\quad &&\forall \boldsymbol{\eta}\in [V_K]^d\quad\forall F \in \mathcal{F}_K^I\\
\int_{\Omega} l_F(v)\cdot\boldsymbol{\eta}\ dx &=-\int_{F} v\llbracket\boldsymbol{\eta}\rrbracket\ ds\quad &&\forall \boldsymbol{\eta}\in [V_K]^d\quad\forall F \in \mathcal{F}_K^I,
\end{alignat*}
and set
\begin{equation}
\label{RLdef}
\mathcal{R}_K(\boldsymbol{\tau}) = \sum_{F\in\mathcal{F}_K} r_F(\boldsymbol{\tau}),\qquad\mathcal{L}_K(v) = \sum_{F\in\mathcal{F}_K} l_F(v).
\end{equation}

The DG finite element formulation reads as follows: find $u_K \in V_K$ such that
\begin{equation}
\mathcal{A}_K(u_K,v_K)=\int_\Omega f v_K\ dx \quad \forall v_K\in V_K,
\label{DG}
\end{equation}
with $\mathcal{A}_K(\cdot,\cdot): V_K\times V_K\rightarrow \mathbb{R}$ defined as
\begin{align}
\label{bilinearK}
\mathcal{A}_K(w,v) = &\sum_{T\in\mathcal{T}_K}\int_T \nabla w \cdot \nabla v\ dx+\sum_{T\in\mathcal{T}_K}\int_T \nabla w \cdot (\mathcal{R}_K(\llbracket v \rrbracket)+\mathcal{L}_K(\boldsymbol{\beta}\cdot\llbracket v \rrbracket))\ dx\notag\\
&+\sum_{T\in\mathcal{T}_K}\int_T (\mathcal{R}_K(\llbracket w \rrbracket)+\mathcal{L}_K(\boldsymbol{\beta}\cdot\llbracket w \rrbracket))\cdot\nabla v\ dx+\mathcal{S}_K^j(w,v)\\
&+\theta\int_\Omega (\mathcal{R}_K(\llbracket w \rrbracket)+\mathcal{L}_K(\boldsymbol{\beta}\cdot\llbracket w \rrbracket))\cdot(\mathcal{R}_K(\llbracket v \rrbracket)+\mathcal{L}_K(\boldsymbol{\beta}\cdot\llbracket v \rrbracket))\ dx\notag,
\end{align}
where $\theta=0$ for the SIPG \cite{Arn82} and SIPG($\delta$) \cite{HeinPie} methods and $\theta = 1$ for the LDG method \cite{CoShu}. 
For the LDG method, $\boldsymbol{\beta}\in\mathbb{R}^d$ is a uniformly bounded (and possibly null) vector, whereas for the SIPG($\delta$) and SIPG methods $\boldsymbol{\beta}=\boldsymbol{\delta}$, and $\boldsymbol{\beta}=\boldsymbol{0}$, respectively.
The stabilization term $\mathcal{S}_K^j(\cdot,\cdot)$  is defined as
\begin{equation}
\label{jumpjump}
\mathcal{S}_K^j(w,v) = \sum_{F\in \mathcal{F}_K}\int_{F} \sigma_K\llbracket w \rrbracket\cdot \llbracket v\rrbracket\ ds\quad \forall w,v\in V_K
\end{equation}
where the penalization term $\sigma_K\in L^\infty(\mathcal{F}_K)$ is chosen such that
\begin{align}\label{eq:penalty}
&\sigma_K|_F = 
\frac{\alpha_K p_K^2}{\min(h_{T^+},h_{T^{-}})},
\quad F\in \mathcal{F}_K^I,
&\sigma_K|_F = 
\frac{\alpha_K p_K^2}{h_T}
\quad F\in \mathcal{F}_K^B,
\end{align}
with $\alpha_K\in \mathbb{R}^+$ and $h_{T^{\pm}}$ diameters of the neighboring elements $T^{\pm} \in \mathcal{T}_K$.

We endow the space $V_K$ with the DG norm $\| \cdot\|_{\textnormal{DG},K}$ defined as
\begin{equation}
\| v \|_{\textnormal{DG},K}^2 = \sum_{T\in\mathcal{T}_K} \|\nabla v\|^2_{L^2(T)}+\sum_{F\in \mathcal{F}_K}\|\sigma_K^{1/2}\llbracket v \rrbracket\|_{L^2(F)}^2.
\label{DGnorm}
\end{equation}
The following result ensures the well posedness of problem \eqref{DG}, cf. e.g., \cite{HouSchSul,PerSchot,AntHou,StaWihl}.
\begin{lemma}
\label{lem:contcoerc}
Let $V(\h_K) = V_K+V$. It holds
\begin{alignat}{2}
\mathcal{A}_K(u,v)&\lesssim \| u \|_{\textnormal{DG},K}\| v \|_{\textnormal{DG},K}\quad &\forall u,v\in V(\h_K), \label{cont}\\
\mathcal{A}_K(u,u)&\gtrsim \| u \|_{\textnormal{DG},K}^2 &\forall u\in V_K.\label{coerc}
\end{alignat}
For the SIPG and SIPG($\delta$) methods, coercivity holds provided the stabilization parameter $\alpha_K$ is chosen large enough.
\end{lemma}\par
Since the bilinear form \eqref{bilinearK} contains the lifting operators, continuity in $V(\h_K)$ and coercivity in $V_K$ can be proved with respect to the same DG norm \eqref{DGnorm}. This is different from the approach proposed in \cite{Arn}, where continuity holds in $V(\h_K)$ in an augmented norm.\\

We have the following error estimates,
cf. \cite{PerSchot,HouSchSul,StaWihl}.
\begin{theorem}
\label{thm:DGerr}
Let $u$ be the exact solution of problem \eqref{weak} such that $u\in H^{s+1}(\mathcal{T}_K)$, $s \geq 1$, and let $u_K\in V_K$ be the DG solution of problem \eqref{DG}. Then,
\begin{align}
\|u-u_K\|_{\textnormal{DG},K}&\lesssim \frac{h_K^{\min(p_K,s)}}{p_K^{s-1/2}}\|u\|_{H^{s+1}(\mathcal{T}_K)}\label{DGerror1},\\
\|u-u_K\|_{L^2(\Omega)}&\lesssim \frac{\h_K^{\min(p_K,s)+1}}{p_K^{s}}\|u\|_{H^{s+1}(\mathcal{T}_K)}.\label{L2error1}
\end{align}
\end{theorem}
The proof of Theorem~ \ref{thm:DGerr} follows the lines given in \cite{PerSchot}; for the sake of completeness we sketch it in Appendix~\ref{appendixA}.
\begin{remark}
\label{rem:err}
Optimal error estimates with respect to $p_K$
can be shown using the projector of \cite{GeorgSul05} provided the solution belongs to a suitable augmented space, or whenever a continuous interpolant can be built, cf. \cite{StaWihl}. Therefore, in the following we will write
\begin{equation}\label{error_estimates}
\begin{aligned}
\|u-u_K\|_{\textnormal{DG},K}&\lesssim \frac{h_K^{\min(p_K,s)}}{p_K^{s-\mu/2}}\|u\|_{H^{s+1}(\mathcal{T}_K)},\\
\|u-u_K\|_{L^2(\Omega)}&\lesssim \frac{\h_K^{\min(p_K,s)+1}}{p_K^{s+1-\mu}}\|u\|_{H^{s+1}(\mathcal{T}_K)},
\end{aligned}
\end{equation}
with $\mu = 0,1$ for optimal and suboptimal $p_K$ estimates, respectively.
\end{remark}
\section{Multigrid W-cycle methods with non-inherited sublevel solvers}
\label{mulgr}
Before introducing our W-cycle algorithms, we make some further assumptions and introduce some notation.
We suppose that the grid $\mathcal{T}_K$ has been obtained by $K-1$ successive uniform refinements  using the red-green algorithm of an initial (coarse) quasi-uniform partition $\mathcal{T}_1$. More precisely, for $d=2$, given the initial mesh $\mathcal{T}_1$ of size $\h_1$, the grid $\mathcal{T}_k$, $k=2,\dots,K$, is built by splitting each triangle/parallelogram of $\mathcal{T}_{k-1}$ into four congruent triangles/parallelograms connecting the midpoints of opposite edges, thus obtaining a mesh with size $\h_k = \h_12^{1-k}$. If $d=3$ each element is splitted into eight tetrahedra/parallelepipeds. The associated discontinuous
spaces $V_1\subseteq V_2\subseteq\dots\subseteq V_K$ are defined according to \eqref{VK}
$$V_k:=\{v\in L^2(\Omega):v\circ \mathsf{F}_T \in \mathbb{M}^{p_k}(\widehat{T})\quad \forall T\in \mathcal{T}_k\},\quad k=1,\dots,K.$$
Analogously $V(\h_k) = V_k+V$.
We will assume that
\begin{equation}
\label{pkpk1}
p_{k-1}\leq p_{k}\lesssim p_{k-1}\qquad \forall k = 2,\dots, K,
\end{equation}
that is when varying from one mesh level to another the polynomial approximation degrees vary moderately.
Let $n_k$ be the dimension of $V_k$, and let $\{  \phi^k_i \}_{i=1}^{n_k}$ be a set of basis functions of $V_k$. Any $v\in V_k$ can then be written as
$$v=\sum_{i=1}^{n_k}v_i \phi^k_i,\quad v_i\in\mathbb{R},\quad i = 1,\dots,n_k.$$
We will suppose that $\{\phi^k_i\}_{i=1}^{n_k}$ is a set of basis functions which are orthonormal with respect to the $L^2(\widehat{T})$-inner product, being $\widehat{T}$ the reference element. A detailed construction of such a basis can be found in \cite{AntHou}.
On $V_k$ we then introduce the mesh-dependent inner product
\begin{equation}
\label{innerprod}
(u,v)_k=\h_k^d\sum_{i=1}^{n_k}u_iv_i\quad \forall u,v,\in V_k,\quad u_i,v_j\in \mathbb{R},\ i,j=1,\dots,n_k.
\end{equation}
The next result establishes the connection between \eqref{innerprod} and the $L^2$ norm.
\begin{lemma}
For any $v\in V_k$, $k=1,\dots,K$, it holds 
\begin{equation}
\label{equivprod}
(v, v)_{k}\lesssim \left\|v\right\|_{L^2(\Omega)}^2 \lesssim (v, v)_{k}.
\end{equation}
\end{lemma}
\begin{proof}
Let $v\in V_k$, we write $v = \sum_{i=1}^{n_k} v_i\phi^k_i$ and
\begin{equation*}
\|v\|_{L^2(T)}^2= \int_T \left(\sum_{i=1}^{n_k} v_i\phi^k_i\right)\left(\sum_{j=1}^{n_k} v_j\phi^k_j\right)\ dx 
= \sum_{i,j=1}^{n_k} v_i v_j\int_T \phi^k_i\phi^k_j\ dx
= \sum_{i=1}^{n_k} v_i^2
\|\phi^k_i\|_{L^2(T)}^2,
\end{equation*}
where in the last step we have used a scaling argument and the fact that the basis functions $\{  \phi^k_i \}_{i=1}^{n_k}$ are $L^2$-orthogonal on the reference element. Using that $\h_k^d\lesssim\|\phi^k_i\|_{L^2(T)}\lesssim \h_k^d $ (cf. \cite[Proposition 3.4.1]{QuaVal})
the thesis follows.
\end{proof}\par\medskip
Once the basis of $V_K$ is chosen, equation \eqref{DG} can be written as the following linear system of equations
\begin{equation}
\label{system}
A_Ku_K=f_K,
\end{equation}
where the operators $A_K:V_K\rightarrow V_K'$ and $f_K\in V_K'$ are defined as
\begin{equation}\label{defAK}\begin{aligned}
&(A_K u,v)_K=\mathcal{A}_K(u,v), 
&(f_K,v)_K=\int_\Omega f v\ dx
&&\forall u,v\in V_K,
\end{aligned}
\end{equation}
being $V_k'$ the dual of $V_k$.
In order to define the subproblems on the coarse levels $k=1,\ldots, K-1$, we consider the corresponding bilinear forms $\mathcal{A}_k(\cdot,\cdot): V_k\times V_k\rightarrow \mathbb{R}$, cf. \eqref{bilinearK}
\begin{align}
\label{bilineark}
\mathcal{A}_{k} (u,v) = &\sum_{T\in\mathcal{T}_k}\int_T \nabla u \cdot \nabla v\ dx+\sum_{T\in\mathcal{T}_k}\int_T \nabla u \cdot (\mathcal{R}_k(\llbracket v \rrbracket)+\mathcal{L}_k(\boldsymbol{\beta}\cdot\llbracket v \rrbracket))\ dx\notag\\
&+\sum_{T\in\mathcal{T}_k}\int_T (\mathcal{R}_k(\llbracket u \rrbracket)+\mathcal{L}_k(\boldsymbol{\beta}\cdot\llbracket u \rrbracket))\cdot\nabla v\ dx+\mathcal{S}_k^j(u,v)\\
&+\theta\sum_{T\in\mathcal{T}_k}\int_T (\mathcal{R}_k(\llbracket u \rrbracket)+\mathcal{L}_k(\boldsymbol{\beta}\cdot\llbracket u \rrbracket))\cdot(\mathcal{R}_k(\llbracket v \rrbracket)+\mathcal{L}_k(\boldsymbol{\beta}\cdot\llbracket v \rrbracket))\ dx,\notag
\end{align}
where 
\begin{equation*}
\mathcal{S}_k^j(w,v) = \sum_{F\in \mathcal{F}_k}\int_{F} \sigma_k\llbracket w \rrbracket\cdot \llbracket v\rrbracket\ ds\quad \forall w,v\in V_k,
\end{equation*}
cf. \eqref{jumpjump}, and where $\sigma_k\in L^\infty(\mathcal{F}_k)$ is defined according to \eqref{eq:penalty}, but on the level $k$.
We then set
\begin{equation}
(A_{k} u,v)_{k}=\mathcal{A}_{k}(u,v)\qquad \forall u,v\in V_{k}.
\label{defAk1}
\end{equation}

Recalling Lemma~\ref{lem:contcoerc}, continuity and coercivity of the bilinear forms $\mathcal{A}_{k} (\cdot,\cdot)$, $k =1,\dots,K-1$, with respect to the DG norms defined on the level $k$ 
\begin{equation}
\label{DGnormk}
\|v\|_{\textnormal{DG},k}^2 = \sum_{T\in\mathcal{T}_k} \|\nabla v\|^2_{L^2(T)}+\sum_{F\in\mathcal{F}_k} \| \sigma^{1/2}_k\llbracket v\rrbracket\|^2_{L^2(F)},
\end{equation}
easily follow, {\it i.e.},
\begin{alignat}{2}
\mathcal{A}_k(u,v)&\lesssim \| u \|_{\textnormal{DG},k}\| v \|_{\textnormal{DG},k}\quad &&\forall u,v\in V(\h_k), \label{contk}\\
\mathcal{A}_k(u,u)&\gtrsim \| u \|_{\textnormal{DG},k}^2 &&\forall u\in V_k.\label{coerck}
\end{alignat}
Moreover, since it holds that
\begin{equation}
\label{hest}
\h_{k}\leq\h_{k-1}\lesssim\h_{k}\quad\forall k = 2,\dots,K,
\end{equation}
and thanks to hypothesis \eqref{pkpk1}, it also follows that
\begin{equation}
\label{DGequiv}
\| v_{k-1} \|_{\textnormal{DG},k-1}\leq \| v_{k-1} \|_{\textnormal{DG},k}\lesssim \frac{p_{k}}{p_{k-1}} \| v_{k-1} \|_{\textnormal{DG},k-1}\lesssim \| v_{k-1} \|_{\textnormal{DG},k-1},
\end{equation}
for any $v\in V_{k-1}$, $k = 2,\dots,K$. 
The hidden constant in the above inequality  depends on the ratio $p_{k}/p_{k-1}$, which means that in absence of the assumption \eqref{pkpk1}, such a dependence should be taken into account.\\

To introduce our multigrid algorithm, we need two ingredients: intergrid transfer operators ({\it  restriction} and {\it  prolongation}) and a smoothing iteration. The prolongation operator connecting the space $V_{k-1}$ to $V_{k}$ is denoted by $R^{k}_{k-1}:V_{k-1}\rightarrow V_{k}$, while the restriction operator is the adjoint with respect to the discrete inner product \eqref{innerprod} and is denoted by 
$R^{k-1}_{k}:V_{k}\rightarrow V_{k-1}$, {\it  i.e.},
$$(R_{k-1}^{k}v,w)_{k}  = (v,R^{k-1}_{k}w)_{k-1}\qquad \forall v\in V_{k-1},w\in V_{k}.$$
We next define the operator $P_{k}^{k-1}:V_{k}\rightarrow V_{k-1}$ as
\begin{equation}
\mathcal{A}_{k-1}(P_{k}^{k-1}v,w) = \mathcal{A}_{k}(v,R^{k}_{k-1}w)\qquad \forall v\in V_{k}, w\in V_{k-1}.
\label{project}
\end{equation}
For the smoothing scheme, we choose a Richardson iteration, given by:
\begin{equation}
\label{defsmooth}
B_k = \Lambda_k\textnormal{I}_k,
\end{equation}
where $\textnormal{I}_k$ is the identity operator and $\Lambda_k\in\mathbb{R}$ represents a bound for the spectral radius of $A_k$. According to \cite[Lemma 2.6]{AntHou} and using the equivalence \eqref{equivprod}, the following estimate for the maximum eigenvalue of $A_k$ can be shown
\begin{equation}
\label{eigA}
\lambda_{max}(A_k)\lesssim \frac{p_k^4}{h_k^{2}},
\end{equation}
hence,
\begin{equation}
\label{Lambda}
\Lambda_k\lesssim \frac{p_k^4}{h_k^{2}}.
\end{equation}
Let us now consider the linear system of equations on level $k$
\begin{equation*}
A_k z=g,
\end{equation*}
with a given $g\in V_k'$. We denote by $\mathsf{MG}_\mathcal{W} (k,g,z_0,m_1,m_2)$ the approximate solution obtained by applying the $k$-th level iteration to the above linear system, with initial guess $z_0$ and using $m_1$, $m_2$ number of pre- and post-smoothing steps, respectively. For $k = 1$, (coarsest level) the solution is computed with a direct method, that is
$$\mathsf{MG}_\mathcal{W} (1,g,z_0,m_1,m_2) = A_1^{-1}g,$$
while for $k>1$ we adopt the recursive procedure described in Algorithm~\ref{multilevel}.
\begin{algorithm}[H]
\caption{Multigrid W-cycle scheme}
\label{multilevel}
\begin{algorithmic}
\State \underline{{\it  Pre-smoothing}}:
\For{$i=1,\dots,m_1$} \State $z^{(i)}=z^{(i-1)}+B_k^{-1}(g-A_kz^{(i-1)});$ \EndFor\vspace{0.3cm}
\State \underline{{\it  Coarse grid correction}}:
\State $r_{k-1} = R_k^{k-1}(g-A_kz^{(m_1)})$;
\State $\overline{e}_{k-1} = \mathsf{MG}_\mathcal{W} (k-1,r_{k-1},0,m_1,m_2)$;
\State $e_{k-1} = \mathsf{MG}_\mathcal{W} (k-1,r_{k-1},\overline{e}_{k-1},m_1,m_2)$;
\State $z^{(m_1+1)}=z^{(m_1)}+R_{k-1}^ke_{k-1}$;\vspace{0.3cm}
\State \underline{{\it  Post-smoothing}}:
\For{$i=m_1+2,\dots,m_1+m_2+1$} \State $z^{(i)}=z^{(i-1)}+B_k^{-1}(g-A_kz^{(i-1)});$ \EndFor\vspace{0.3cm}
\State $\mathsf{MG}_\mathcal{W} (k,g,z_0,m_1,m_2)=z^{(m_1+m_2+1)}.$
\end{algorithmic}
\end{algorithm}

We now introduce the error propagation operator 
$$\mathbb{E}_{k,m_1,m_2} (z-z_0)=z-\mathsf{MG}_\mathcal{V}(k,g,z_0,m_1,m_2),$$
and recall that, according to \cite{Hack,bramble},  the following relation holds
\begin{equation}
\begin{cases} 
\mathbb{E}_{1,m_1,m_2}v &= 0\\
\mathbb{E}_{k,m_1,m_2}v &= G_k^{m_2}(\textnormal{I}_k-R_{k-1}^k(\textnormal{I}_k-\mathbb{E}_{k-1,m_1,m_2}^2)P_k^{k-1})G_k^{m_1}v\qquad k>1,
\end{cases}
\label{defEk}
\end{equation}
where $G_k = \textnormal{I}_k-B_k^{-1}A_k$ satisfies
\begin{equation*}
\mathcal{A}_k(G_kv,w) = \mathcal{A}_k(v,G_kw)\quad \forall v,w\in V_k.
\label{Gsymm}
\end{equation*}
Indeed, using the definition of $G_k$ and that $B_k = \Lambda_k\textnormal{I}_k$, cf.\eqref{defsmooth},
\begin{align*}
\mathcal{A}_k(G_kv,w) &= \mathcal{A}_k(v,w)-\frac{1}{\Lambda_k}(A_kv,A_kw)_k
&=\mathcal{A}_k(v,w)-\mathcal{A}_k(v,\frac{A_k}{\Lambda_k}w)=\mathcal{A}_k(v,G_kw).
\end{align*}

\section{Convergence analysis of the W-cycle multigrid method}
\label{convergence}
To prove convergence, we need to obtain an estimate for $\mathbb{E}_{k,m_1,m_2}$ in a proper norm. We then define 
\begin{equation}
\label{discrnorm}
\ltrivert v\rtrivert_{s,k}=\sqrt{(A_k^sv,v)_k}\qquad \forall s\in\mathbb{R},\ v\in V_k,\quad k=1,\dots,K,
\end{equation}
and observe that 
\begin{equation}
\label{norm1Ak}
\ltrivert v\rtrivert_{1,k}^2 = 
\sqrt{(A_kv,v)_k}=\mathcal{A}_k(v,v)\quad \ltrivert v\rtrivert_{0,k}^2 =(v,v)_k\quad\forall v\in V_k,
\end{equation}
and, by virtue of \eqref{equivprod}, it holds that
\begin{equation}
\label{equivL2}
\ltrivert v\rtrivert_{0,k}\lesssim \left\|v\right\|_{L^2(\Omega)} \lesssim \ltrivert v\rtrivert_{0,k}.
\end{equation}
In the following we will often make use of the eigenvalue problem associated to $A_k$
\begin{equation}
\label{eig}
A_k\psi^k_i=\lambda_i \psi^k_i,
\end{equation}
where $0<\lambda_1\leq\lambda_2\leq\dots\leq\lambda_{n_k}$ represent the eigenvalues of $A_k$ and 
$\{ \psi^k_i\}_{i=1}^{n_k}$ are the associated eigenvectors which form a basis for $ V_k$. We can then write any $v\in V_k$ as
\begin{equation}
\label{base}
 v=\sum_{i=1}^{n_k} v_i\psi^k_i,\qquad v_i\in\mathbb{R}.
\end{equation}
Next, we introduce the following generalized Cauchy-Schwarz inequality.
\begin{lemma}
\label{lem:Csgen}
For any $v,w\in V_k$ and $s\in \mathbb{R}$, it holds
\begin{equation}
\mathcal{A}_k(v,w)\leq \ltrivert v \rtrivert_{1+s,k}\ltrivert w\rtrivert_{1-s,k}.
\label{Csgen}
\end{equation}
\end{lemma}
\begin{proof}
Considering the eigenvalue problem \eqref{eig} and the relation \eqref{base}, it follows that
$$A_kv = \sum_{i=1}^{n_k} v_iA_k\psi^k_i = \sum_{i=1}^{n_k} v_i\lambda_i\psi^k_i\qquad \forall v\in V_k.$$
From the definition \eqref{bilineark} of  $A_k$ and of the inner product \eqref{innerprod}, it follows 
$$\mathcal{A}_k(v,w)=(A_kv,w)_k = \h_k^d\sum_{i=1}^{n_k} v_i w_i\lambda_i=\h_k^d\sum_{i=1}^{n_k} v_i\lambda_i^{\frac{1+s}{2}} w_i\lambda_i^{\frac{1-s}{2}}.$$
The thesis follows applying the Cauchy-Schwarz inequality
\begin{align*}
(A_kv,w)_k = \h_k^d\sum_{i=1}^{n_k} v_i\lambda_i^{\frac{1+s}{2}} w_i\lambda_i^{\frac{1-s}{2}}&\leq \sqrt{\h_k^d\sum_{i=1}^{n_k} v_i^2\lambda_i^{1+s}}\sqrt{\h_k^d\sum_{j=1}^{n_k} w_j^2\lambda_j^{1-s}}=\ltrivert v \rtrivert_{s+1,k}\ltrivert w \rtrivert_{s-1,k}.
\end{align*}
\end{proof}\par

To establish an estimate for the error propagation operator, we follow a standard approach by separating two contributions: the {\it  smoothing property} and the {\it  approximation property}. The {\it  smoothing property} pertains only the smoothing scheme chosen for the multigrid algorithm.
\begin{lemma}[Smoothing property]
\label{lem:smooth1}
For any $v\in V_k$, it holds
\begin{equation}
\begin{aligned}
& \ltrivert G_k^mv\rtrivert_{s,k} \lesssim p_k^{2(s-t)}h_k^{t-s}(1+m)^{(t-s)/2}\ltrivert v \rtrivert_{t,k}, && 0\leq t\leq s\leq 2.
\end{aligned}
\label{smoothingprop1}
\end{equation}
\end{lemma}
\begin{proof}
We refer again to the eigenvalue problem \eqref{eig} and write $v$ according to \eqref{base}:
\begin{equation*}
G_k^m v = (\textnormal{I}_k-\frac{1}{\Lambda_k}A_k)^mv = \sum_{i=1}^{n_k}(1-\frac{\lambda_i}{ \Lambda_k})^mv_i\psi_i^k.
\end{equation*}
The thesis follows using the above identity and estimate \eqref{Lambda}
\begin{align*}
\ltrivert G_k^m v \rtrivert_{s,k}^2 &= \h_k^d\sum_{i=1}^{n_k}\left(1-\frac{\lambda_i}{\Lambda_k}\right)^{2m}v_i^2\lambda_i^s=\Lambda_k^{s-t} \left\{\h_k^d\sum_{i=1}^{n_k}\left(1-\frac{\lambda_i}{\Lambda_k}\right)^{2m}\frac{\lambda_i^{s-t}}{\Lambda_k^{s-t}}v_i^2\lambda_i^t\right\}\\
&\leq \Lambda_k^{s-t} \max_{x\in[0,1]} \{x^{s-t}(1-x)^{2m}\}\ltrivert v\rtrivert_{t,k}^2\leq p_k^{4(s-t)}h_k^{2(t-s)} (1+m)^{t-s}\ltrivert v\rtrivert_{t,k}^2.
\end{align*}
\end{proof}\par
Following \cite[Lemma 4.2]{BrenZhao}, we now prove the {\it  approximation property}.
\begin{lemma}[Approximation property]
\label{lem:approx}
Let $\mu$ be defined as in Remark~\ref{rem:err}. Then,
\begin{equation}
\begin{aligned}
& \ltrivert (\textnormal{I}_k-R_{k-1}^kP_k^{k-1})v\rtrivert_{0,k}\lesssim \frac{\h_{k-1}^2}{p_{k-1}^{2-\mu}}\ltrivert v\rtrivert_{2,k}
&& \forall v \in V_k.
\end{aligned}
\label{approxprop}
\end{equation}
\end{lemma}
\begin{proof}
For any $v\in V_k$, applying \eqref{equivL2} and the duality formula for the $L^2$ norm, we obtain 
\begin{align}
\label{dualratio}
\ltrivert (\textnormal{I}_k-R_{k-1}^kP_k^{k-1})v\rtrivert_{0,k} &\lesssim \| (\textnormal{I}_k-R_{k-1}^kP_k^{k-1})v\|_{L^2(\Omega)}
&= \sup_{
\substack{
\phi \in L^2(\Omega)\\
\phi  \neq 0}} 
\frac{\int_{\Omega}\phi(\textnormal{I}_k-R_{k-1}^kP_k^{k-1})v\ dx}{\| \phi\|_{L^2(\Omega)}}.
\end{align}
Next, for $\phi\in L^2(\Omega)$, let $\eta\in H^2(\Omega)\cap  H^1_0(\Omega)$
be the solution to 
\begin{equation*}
\begin{aligned}
\int_{\Omega} \nabla \eta\cdot\nabla v\ dx &= \int_{\Omega}\phi v\ dx &&\forall v\in H^1_0(\Omega),
\end{aligned}
\end{equation*}
and let $\eta_k\in V_k$ and $\eta_{k-1}\in V_{k-1}$ be its 
DG approximations  in $V_k$ and $V_{k-1}$
\begin{equation}
\label{etak}
\begin{aligned}
 \mathcal{A}_k(\eta_k,v)&= \int_{\Omega}\phi v\ dx
&&\forall v\in V_k,\\
 \mathcal{A}_{k-1}(\eta_{k-1},v)&= \int_{\Omega}\phi v\ dx 
&&\forall v\in V_{k-1}.
\end{aligned}
\end{equation}
By \eqref{error_estimates}, assumption \eqref{elliptic}, the hypotheses \eqref{hest} and \eqref{pkpk1} we have
\begin{align}
\| \eta-\eta_k \|_{L^2(\Omega)}\lesssim \frac{\h_{k-1}^2}{p_{k-1}^{2-\mu}}\| \phi\|_{L^2(\Omega)},\quad\| \eta-\eta_{k-1} \|_{L^2(\Omega)}\lesssim \frac{\h_{k-1}^2}{p_{k-1}^{2-\mu}}\| \phi\|_{L^2(\Omega)}.\label{dualerrL2}
\end{align}
Moreover, if we consider the definion \eqref{project} of $P_{k}^{k-1}$  and \eqref{etak}, it holds that
$$\mathcal{A}_{k-1} (P_k^{k-1}\eta_k,w ) = \mathcal{A}_{k} (\eta_k,R^k_{k-1}w ) = \mathcal{A}_{k} (\eta_k,w ) = \phi(w) = \mathcal{A}_{k-1} (\eta_{k-1},w )\quad \forall w\in V_{k-1},$$
which implies 
\begin{equation}
\eta_{k-1} = P_k^{k-1}\eta_k.
\label{ekPek1}
\end{equation}
Now applying \eqref{etak}, the definition \eqref{project} of $P^{k-1}_k$, \eqref{ekPek1}, the Cauchy-Schwarz inequality \eqref{Csgen}, the $L^2$ norm equivalence \eqref{equivL2} and the error estimates \eqref{dualerrL2} we obtain
\begin{align*}
\int_\Omega \phi(\textnormal{I}_k-R_{k-1}^kP^{k-1}_k)v\ dx 
=& \mathcal{A}_k(\eta_k,v)-\mathcal{A}_k(\eta_k,R_{k-1}^kP^{k-1}_kv)\\
=& \mathcal{A}_k(\eta_k,v)-\mathcal{A}_{k-1}(P^{k-1}_k\eta_k,P^{k-1}_kv)\\
=& \mathcal{A}_k(\eta_k,v)-\mathcal{A}_{k-1}(\eta_{k-1},P^{k-1}_kv)
= \mathcal{A}_k(\eta_k-R_{k-1}^k\eta_{k-1},v)\\
\leq&\ltrivert \eta_k-\eta_{k-1} \rtrivert_{0,k}\ltrivert v \rtrivert_{2,k}
\lesssim \| \eta_k-\eta_{k-1} \|_{L^2(\Omega)}\ltrivert v \rtrivert_{2,k}\\
\leq & (\| \eta_k-\eta \|_{L^2(\Omega)}+\| \eta_{k-1}-\eta \|_{L^2(\Omega)})\ltrivert v \rtrivert_{2,k}\\
\lesssim &\frac{\h_{k-1}^2}{p_{k-1}^{2-\mu}}\| \phi \|_{L^2(\Omega)}\ltrivert v \rtrivert_{2,k}.
\end{align*}
The above estimate together with \eqref{dualratio} gives the desired inequality.
\end{proof}

Lemma~\ref{lem:smooth1} and Lemma~\ref{lem:approx} allow the convergence anaylsis of the two-level method, whose error propagation operator is given by
$$\mathbb{E}^{\textnormal{2lvl}}_{k,m_1,m_2} = G_k^{m_2}(\textnormal{I}_k-R_{k-1}^kP_k^{k-1})G_k^{m_1}.$$
\begin{theorem}
\label{thm: 2lvl}
There exists a  positive constant $C_{\textnormal{2lvl}}$ independent of the mesh size, the polynomial approximation degree and the level $k$, such that
$$\ltrivert\mathbb{E}^{\textnormal{2lvl}}_{k,m_1,m_2}v\rtrivert_{1,k} \leq C_{\textnormal{2lvl}}\Sigma_k \ltrivert v\rtrivert_{1,k}$$
for any $v\in V_k$, with 
\begin{align*}
\Sigma_k=\frac{p_{k}^{2+\mu}}{(1+m_1)^{1/2}(1+m_2)^{1/2}},
\end{align*}
and $\mu$ defined as in Remark~\ref{rem:err}.
Therefore, the two-level method converges provided the number of pre-smoothing and post-smoothing steps is chosen large enough.
\end{theorem}
\begin{proof}
Exploiting the smoothing property \eqref{smoothingprop1}, approximation property \eqref{approxprop} and assumptions \eqref{hest} and \eqref{pkpk1}, we obtain
\begin{align*}
\ltrivert\mathbb{E}^{\textnormal{2lvl}}_{k,m_1,m_2}v\rtrivert_{1,k} &=\ltrivert G_k^{m_2}(\textnormal{I}_k-R_{k-1}^kP_k^{k-1})G_k^{m_1}v\rtrivert_{1,k}\\
&\lesssim \h_k^{-1}p_k^2(1+m_2)^{-1/2}\ltrivert (\textnormal{I}_k-R_{k-1}^kP_k^{k-1})G_k^{m_1}v\rtrivert_{0,k}\\
&\lesssim \h_kp_k^2 p_{k-1}^{\mu-2}(1+m_2)^{-1/2}\ltrivert G_k^{m_1}v\rtrivert_{2,k}\\
&\lesssim p_{k}^{2+\mu}(1+m_1)^{-1/2} (1+m_2)^{-1/2}\ltrivert v\rtrivert_{1,k},
\end{align*}
and the proof is complete.
\end{proof}
\begin{remark}
From Theorem \ref{thm: 2lvl} we have that the number of smoothing steps needed for convergence of the two level method increases with the polynomial approximation degree. Indeed,
let $\widehat{\mathsf{m}}_{\textnormal{2lvl}} \leq m_1+m_2$ to be chosen, it follows
$$C_{\textnormal{2lvl}} p_{k}^{2+\mu}(1+m_1)^{-1/2} (1+m_2)^{-1/2}\leq C_{\textnormal{2lvl}} p_{k}^{2+\mu}\widehat{\mathsf{m}}_{\textnormal{2lvl}}^{-1/2}.$$
The value of $\widehat{\mathsf{m}}_{\textnormal{2lvl}}$ is then chosen in such a way that 
$$C_{\textnormal{2lvl}} p_{k}^{2+\mu}\widehat{\mathsf{m}}_{\textnormal{2lvl}}^{-1/2}<1,$$
that is,
$$\widehat{\mathsf{m}}_{\textnormal{2lvl}}^{1/2}>C_{\textnormal{2lvl}} p_{k}^{2+\mu}.$$
\end{remark}\par
The next result regards the stability of the intergrid transfer operator $R_{k-1}^k$ and the operator $P_k^{k-1}$.
\begin{lemma}
\label{lem: RPstab}
There exists a positive constant $C_{\textnormal{stab}}$ independent of the mesh size, the polynomial approximation degree and the level $k$, such that
\begin{align}
\ltrivert R_{k-1}^k v\rtrivert_{1,k} &\leq C_{\textnormal{stab}}\ltrivert  v\rtrivert_{1,k-1}
&& \forall v\in V_{k-1},\label{Rstab}\\
\ltrivert P^{k-1}_k v\rtrivert_{1,k-1} &\leq C_{\textnormal{stab}}\ltrivert  v\rtrivert_{1,k}&&
\forall v\in V_{k}.\label{Pstab}
\end{align}
\end{lemma}
\begin{proof}
We apply \eqref{norm1Ak}, the continuity bound \eqref{cont}, the relation \eqref{DGequiv} between the DG norms on different levels and the coercivity bound \eqref{coerck}
\begin{align*}
\ltrivert R_{k-1}^k v\rtrivert_{1,k}^2 &= \mathcal{A}_k(R_{k-1}^k v,R_{k-1}^k v)\lesssim \| R_{k-1}^k v\|_{\textnormal{DG},k}^2\lesssim\| v\|_{\textnormal{DG},k-1}^2\\
&\lesssim \mathcal{A}_{k-1}(v,v) = C_{\textnormal{stab}}^2\ltrivert v\rtrivert_{1,k-1}^2.
\end{align*} 
Inequality \eqref{Pstab} is obtained by the definition \eqref{project} of $P^{k-1}_k$, \eqref{norm1Ak} and \eqref{Rstab}
\begin{align*}
\ltrivert P^{k-1}_k v\rtrivert_{1,k-1} &= \max_{u\in V_{k-1}\setminus\{0\}}\frac{\mathcal{A}_{k-1}(P^{k-1}_k v,u)}{\ltrivert u\rtrivert_{1,k-1}}=\max_{u\in V_{k-1}\setminus\{0\}}\frac{\mathcal{A}_{k}(v,R_{k-1}^ku)}{\ltrivert u\rtrivert_{1,k-1}}\\&\leq C_{\textnormal{stab}}\frac{\ltrivert v\rtrivert_{1,k}\ltrivert u\rtrivert_{1,k-1}}{\ltrivert u\rtrivert_{1,k-1}}\leq C_{\textnormal{stab}}\ltrivert v\rtrivert_{1,k}.
\end{align*} 
\end{proof}\par
We are now ready to prove the main result of the paper concerning the convergence of the W-cycle multigrid method.
\begin{theorem}
\label{thm:final}
Let $\Sigma_k$ be defined as in Theorem~\ref{thm: 2lvl}. Then, there exist a constant $\widehat{\mathsf{C}}>C_{\textnormal{2lvl}}$ and an integer $\widehat{\mathsf{m}}_k$ independent of the mesh size, but dependent on the polynomial approximation degree, such that 
\begin{equation}
\label{final}
\ltrivert \mathbb{E}_{k,m_1,m_2}v\rtrivert_{1,k}\leq\widehat{\mathsf{C}}\Sigma_k\ltrivert v\rtrivert_{1,k}\quad \forall v\in V_k,
\end{equation}
provided $m_1+m_2\geq \widehat{\mathsf{m}}_k$.
\end{theorem}
\begin{proof}
We follow the guidelines given in \cite[Theorem 4.6]{BrennCuiSu} and proceed by induction. For $k = 1$, \eqref{final} is trivially true. For $k>1$ we assume that \eqref{final} holds for $k-1$. By definition \eqref{defEk} we write $\mathbb{E}_{k,m_1,m_2}v$ as
$$\mathbb{E}_{k,m_1,m_2}v=G_k^{m_2}(\textnormal{I}_k-R_{k-1}^kP_k^{k-1})G_k^{m_1}v+G_k^{m_2}R_{k-1}^k\mathbb{E}_{k-1,m_1,m_2}^2P_k^{k-1}G_k^{m_1}v,$$
hence
$$\ltrivert \mathbb{E}_{k,m_1,m_2}v\rtrivert_{1,k}\leq \ltrivert \mathbb{E}^{\textnormal{2lvl}}_{k,m_1,m_2}v\rtrivert_{1,k}+\ltrivert G_k^{m_2}R_{k-1}^k\mathbb{E}_{k-1,m_1,m_2}^2P_k^{k-1}G_k^{m_1}v\rtrivert_{1,k}.$$
The first term can be bounded by Theorem~\ref{thm: 2lvl}
$$\ltrivert \mathbb{E}^{\textnormal{2lvl}}_{k,m_1,m_2} v\rtrivert_{1,k} \leq C_{\textnormal{2lvl}}\Sigma_k \ltrivert v\rtrivert_{1,k},$$
while the second term is estimated by applying the smoothing property \eqref{smoothingprop1}, the stability estimates \eqref{Rstab} and \eqref{Pstab} and the induction hypothesis
\begin{align*}
\ltrivert G_k^{m_2}R_{k-1}^k\mathbb{E}_{k-1,m_1,m_2}^2P_k^{k-1}G_k^{m_1}v\rtrivert_{1,k}\leq& \ltrivert R_{k-1}^k\mathbb{E}_{k-1,m_1,m_2}^2P_k^{k-1}G_k^{m_1}v\rtrivert_{1,k}\\
\leq&C_{\textnormal{stab}}\ltrivert \mathbb{E}_{k-1,m_1,m_2}^2P_k^{k-1}G_k^{m_1}v\rtrivert_{1,k}\\
\leq&C_{\textnormal{stab}}\widehat{\mathsf{C}}^2\Sigma_{k-1}^2\ltrivert P_k^{k-1}G_k^{m_1}v\rtrivert_{1,k}\\
\leq&C_{\textnormal{stab}}^2\widehat{\mathsf{C}}^2\Sigma_{k-1}^2\ltrivert G_k^{m_1}v\rtrivert_{1,k}\\
\leq&C_{\textnormal{stab}}^2\widehat{\mathsf{C}}^2\Sigma_{k-1}^2\ltrivert v\rtrivert_{1,k}.
\end{align*}
We then obtain
\begin{equation*}
\ltrivert \mathbb{E}_{k,m_1,m_2}v\rtrivert_{1,k}\leq\left(C_{\textnormal{2lvl}}\Sigma_k+C_{\textnormal{stab}}^2\widehat{\mathsf{C}}^2\Sigma_{k-1}^2\right)\ltrivert v\rtrivert_{1,k}.
\end{equation*}
By considering the definition of $\Sigma_k$ given in Theorem~\ref{thm: 2lvl} and \eqref{pkpk1}, we obtain
\begin{align*}
\Sigma_{k-1}^2=\frac{p_{k-1}^{4+2\mu}}{(1+m_1)(1+m_2)}&=\left(\frac{p_{k-1}}{p_{k}}\right)^{2+\mu}\frac{p_{k-1}^{2+\mu}}{(1+m_1)^{1/2}(1+m_2)^{1/2}}\Sigma_{k}\\&\leq\frac{p_{k-1}^{2+\mu}}{(1+m_1)^{1/2}(1+m_2)^{1/2}}\Sigma_{k}.
\end{align*}
Therefore,
\begin{align*}
C_{\textnormal{2lvl}}\Sigma_k+C_{\textnormal{stab}}^2\widehat{\mathsf{C}}^2\Sigma_{k-1}^2\leq \left(C_{\textnormal{2lvl}}+C_{\textnormal{stab}}^2\widehat{\mathsf{C}}^2\frac{p_{k-1}^{2+\mu}}{(1+m_1)^{1/2}(1+m_2)^{1/2}}\right)\Sigma_k.
\end{align*}
We now introduce $\widehat{\mathsf{m}}_k\leq m_1+m_2$, to be chosen later; then it holds
\begin{align*}
C_{\textnormal{2lvl}}+C_{\textnormal{stab}}^2\widehat{\mathsf{C}}^2\frac{p_{k-1}^{2+\mu}}{(1+m_1)^{1/2}(1+m_2)^{1/2}}\leq& C_{\textnormal{2lvl}}+C_{\textnormal{stab}}^2\widehat{\mathsf{C}}^2\frac{p_{k-1}^{2+\mu}}{\widehat{\mathsf{m}}_k^{1/2}}.
\end{align*}
Choosing \begin{equation}
\label{mbound}
\widehat{\mathsf{m}}_k^{1/2} \geq p_{k-1}^{2+\mu}\frac{C_{\textnormal{stab}}^2\widehat{\mathsf{C}}^2}{\widehat{\mathsf{C}}-C_{\textnormal{2lvl}}},
\end{equation}
we obtain
\begin{align*}
C_{\textnormal{2lvl}}\Sigma_k+C_{\textnormal{stab}}^2\widehat{\mathsf{C}}^2\Sigma_{k-1}^2\leq \widehat{\mathsf{C}}\Sigma_k,
\end{align*}
and inequality \eqref{final} follows.
\end{proof}
\section{Extension to other symmetric DG schemes}\label{sec.extesion}

If we restrict to the case of meshes of $d$-parallelepipeds, our analysis can be extended to the methods introduced by Bassi et al. \cite{Basetal} and Brezzi et al. \cite{BreMan} whose bilinear forms can be written as
\begin{multline}
\label{bilinearK2}
\mathcal{A}_K(w,v) = \sum_{T\in\mathcal{T}_K}\int_T \nabla w \cdot \nabla v\ dx+\sum_{T\in\mathcal{T}_K}\int_T \nabla w \cdot \mathcal{R}_K(\llbracket v \rrbracket)\ dx
+\sum_{T\in\mathcal{T}_K}\int_T \mathcal{R}_K(\llbracket w \rrbracket)\cdot\nabla v\ dx\\
+\theta\int_\Omega \mathcal{R}_K(\llbracket w \rrbracket)\cdot\mathcal{R}_K(\llbracket v \rrbracket) dx+\sum_{T\in\mathcal{T}_K}\alpha_K\int_T r_F(\llbracket w\rrbracket)r_F(\llbracket v\rrbracket)\ dx,
\end{multline}
where $\theta = 0,1$ for Bassi et al. \cite{Basetal} and Brezzi et al. methods \cite{BreMan}, repectively.
Indeed, we can prove continuity and coercivity with respect to the DG norm \eqref{DGnorm} using standard techniques and
the following result, cf. \cite{ShotzSchwTos02} for the proof.
\begin{lemma}
\label{lem:restinv}
For any $v\in V_K$ and for any $F \in \mathcal{F}_K$ it holds
\begin{align}
\alpha_K\|r_F(\llbracket v \rrbracket)\|^2_{L^2(\Omega)}\lesssim\|\sqrt{\sigma_K}\llbracket v \rrbracket\|^2_{{L^2}(F)}\lesssim \alpha_K\|r_F(\llbracket v \rrbracket)\|^2_{L^2(\Omega)}.\label{restinv}
\end{align}
\end{lemma}

\section{W-cycle algorithms with inherited bilinear forms}\label{sec:inehrited}

In Section~\ref{convergence} we have followed the classical approach in the framework of multigrid algorithms for DG methods \cite{GopKan03,BrenZhao,BrennCuiSu,BrenCuGuSu}, where the bilinear forms are assembled on each sublevel. We now consider inherited bilinear forms, that is, the sublevel solvers $\mathcal{A}_{k}^{\textnormal{R}} (\cdot,\cdot)$ are obtained as the restriction of the original bilinear form $\mathcal{A}_{K}(\cdot,\cdot)$:
\begin{equation}\label{bilinearR}
\begin{aligned}
&
\mathcal{A}_{k}^{\textnormal{R}} (v,w) = \mathcal{A}_{K} (R^{K}_{k}v,R^{K}_{k}w)
&& \forall v,w\in V_{k}
&& \forall k=1,2,\ldots, K-1.
\end{aligned}
\end{equation}
For $k = 1,\dots,K-1$, the prolongation operators are defined as $R^{K}_{k} = R^{K}_{K-1}R^{K-1}_{K-2}\cdots R^{k+1}_{k}$,  where $R^{k+1}_{k}$ is defined as before. The associated operator $A^\textnormal{R}_k$, given by \eqref{defAk1}, can be computed as
$$A^\textnormal{R}_k = R^k_KA_KR_k^K.$$
Using the new definition of the sublevel solvers, it is easy to see that coercivity estimate remains unchanged, {\it i.e.}, 
$\mathcal{A}^{\textnormal{R}}_k(u,u) \gtrsim \| u \|_{\textnormal{DG},k}^2$ for all $u\in V_k$, 
whereas the  continuity  bound \eqref{cont}  modifies as follows
\begin{equation}\label{contR}
\begin{aligned}
\mathcal{A}^{\textnormal{R}}_k(u,v)&
\lesssim \| u \|_{\textnormal{DG},K}\| v \|_{\textnormal{DG},K}\lesssim \frac{p_K^2}{p_k^2}\frac{h_k}{h_K} \| u \|_{\textnormal{DG},k}\| v \|_{\textnormal{DG},k} &&\forall u,v\in V(\h_k).
\end{aligned}
\end{equation}
The major effect of the above bound regards the estimate \eqref{eigA}, which now becomes
\begin{equation}
\label{eigAR}
\lambda_{max}(A^\textnormal{R}_k)\lesssim \frac{p_K^2p_k^2}{h_Kh_k}.
\end{equation}
Indeed, using the continuity bound \eqref{contR} and estimating separately the contributions of the DG norm $\| \cdot \|_{DG,k}$ we have
\begin{align}
\sum_{T\in\mathcal{T}_{k}} 
\|\nabla u\|^2_{L^2(T)} 
= \sum_{T\in\mathcal{T}_{k}} | u|^2_{H^1(T)}\lesssim \frac{p_{k}^4}{h_{k}^{2}}\sum_{T\in\mathcal{T}_{k}} \| u\|^2_{L^2(T)}
=\frac{p_{k}^4}{h_{k}^{2}} \| u\|^2_{L^2(\Omega)},\\
\sum_{F\in\mathcal{F}_{k}} \| \sigma^{1/2}_{k}\llbracket u\rrbracket\|^2_{L^2(F)}
\lesssim \frac{p_{k}^2}{h_{k}}\sum_{F\in\mathcal{F}_{k}} \| \llbracket u\rrbracket\|^2_{L^2(F)}\lesssim\frac{p_{k}^4}{h_{k}^2} \| u\|^2_{L^2(\Omega)},
\end{align}
where we have also used an inverse inequality and a trace inequality. We then consider the Richardson smoothing scheme with 
\begin{equation}
B^\textnormal{R}_k = \Lambda^\textnormal{R}_k\textnormal{I}_k,
\end{equation}
where, by \eqref{eigAR},  $\Lambda^\textnormal{R}_k\in\mathbb{R}$ is such that
\begin{equation}
\label{LambdaR}
\Lambda_k^R \lesssim \frac{p_K^2p_k^2}{h_Kh_k}.
\end{equation}
We then follow the lines of Section~\ref{mulgr} and Section~\ref{convergence} to define the W-cycle algorithm and analyze its convergence, replacing $A_k$ with $A_k^\textnormal{R}$. We will show that in this case convergence cannot be uniform since it depends on the number of levels. The approximation property of Lemma~\ref{lem:approx} remains trivially true whereas the smoothing operator $G_k$ has to be defined considering $B^\textnormal{R}_k$ instead of $B_k$. As a consequence, the following new 
smoothing property can be proved reasoning as in the proof of Lemma~\ref{lem:smooth1}.
\begin{lemma}
\label{lem:smoothR}
 For $0\leq t\leq s\leq 2$, it holds
\begin{equation}
\begin{aligned}
\ltrivert G_k^mv\rtrivert_{s,k} \lesssim (p_Kp_k)^{(s-t)}(h_Kh_k)^{(t-s)/2}(1+m)^{(t-s)/2}\ltrivert v \rtrivert_{t,k}
&& \forall\, v\in V_k. 
\label{smoothingpropR}
\end{aligned}
\end{equation}
\end{lemma}

\par
Regarding the convergence of the two-level method, estimate \eqref{smoothingpropR} introduces a dependence on the number of levels, as shown in the next result.
\begin{theorem}
\label{thm: 2lvlR}
There exists a  positive constant $C_{\textnormal{2lvl}}^\textnormal{R}$ independent of the mesh size, the polynomial approximation degree and the level $k$, such that
$$\ltrivert\mathbb{E}^{\textnormal{2lvl}}_{k,m_1,m_2}v\rtrivert_{1,k} \leq C_{\textnormal{2lvl}}^\textnormal{R}\Sigma_k^\textnormal{R} \ltrivert v\rtrivert_{1,k}$$
for any $v\in V_k$, with 
\begin{align}
\label{SigmaR}
\Sigma_k^\textnormal{R}=2^{K-k}\frac{p_K^{2}p_{k}^\mu}{(1+m_1)^{1/2}(1+m_2)^{1/2}},
\end{align}
and $\mu$ defined as in Remark~\ref{rem:err}.
\end{theorem}
\par
We observe that the term $2^{K-k}$ in \eqref{SigmaR} is due to the refinement process described in Section~\ref{mulgr}, which implies $\h_{k} = 2^{K-k}\h_K$.
We finally observe that from the definition \eqref{bilinearR},  the stability estimates \eqref{Rstab} and \eqref{Pstab} reduce to
\begin{equation}
\ltrivert R_{k-1}^k v\rtrivert_{1,k} =\ltrivert  v\rtrivert_{1,k-1}
\quad\forall v\in V_{k-1},\qquad
\ltrivert P^{k-1}_k v\rtrivert_{1,k-1} \leq\ltrivert  v\rtrivert_{1,k}
\quad\forall v\in V_{k},
\end{equation}
thus resulting in the following  convergence estimate for the W-cycle algorithm.
\begin{theorem}
\label{thm:finalR}
Let $\Sigma_k^\textnormal{R}$ be defined as in Theorem~\ref{thm: 2lvlR}. Then, there exist a constant $\mathsf{C}^\textnormal{R}>C_{\textnormal{2lvl}}^\textnormal{R}$ and an integer $\mathsf{m}_k^\textnormal{R}$ independent of the mesh size, but dependent on the polynomial approximation degree and the level $k$, such that 
\begin{equation}
\label{finalR}
\ltrivert \mathbb{E}_{k,m_1,m_2}v\rtrivert_{1,k}\leq\mathsf{C}^\textnormal{R}\Sigma_k^\textnormal{R}\ltrivert v\rtrivert_{1,k}\quad \forall v\in V_k,
\end{equation}
provided $m_1+m_2\geq \mathsf{m}_k^\textnormal{R}$ and 
$$ (\mathsf{m}_k^\textnormal{R})^{1/2}\geq 2^{K-k+2}p_K^2p_{k-1}^\mu \frac{(\mathsf{C}^\textnormal{R})^2}{\mathsf{C}^\textnormal{R}-C_{\textnormal{2lvl}}^\textnormal{R}},$$
with $\mu$ defined as in Remark~\ref{rem:err}.
\end{theorem}

\section{Numerical results}
\label{numerical}
In this section we show some numerical results to highlight the practical performance of our W-cycle algorithms. We first verify numerically the smoothing (Lemma~\ref{lem:smooth1}) and approximation (Lemma~\ref{lem:approx}) properties of $h$-multigrid for both the SIPG and LDG methods with a fixed penalization parameter $\alpha_k = \alpha = 10$, $k=1,\dots,K$. We consider a sequence of Cartesian meshes on the unit square $\Omega = [0,1]^2$. Since the dependence on the mesh-size is well known, we restrict ourselves to test the dependence of the smoothing property (with $s=2$ and $t=0$ in \eqref{smoothingprop1}) on the polynomial approximation degree, and the number of smoothing steps.
In the first set of experiments reported in Figure~\ref{fig:smoothp}, we fix $h_K = 0.25$ and $m=2$ and let vary the polynomial approximation order $p_K=p=1,2,\ldots,10$. Figure~\ref{fig:smoothm} shows the analogous results obtained by fixing $\h_K = 0.0625$ and $p=2$ and varying $m$.
We have also estimated numerically the approximation property constant \eqref{approxprop} as a function of the polynomial approximation degree, see Figure~\ref{fig:approxp}. 
We observe that our numerical tests confirm the theoretical estimates given in Section~\ref{mulgr}.\\
\begin{figure}[H]
\label{fig:smoothapprox}
\centering
\subfigure[]
 {\includegraphics[width=0.45\textwidth]{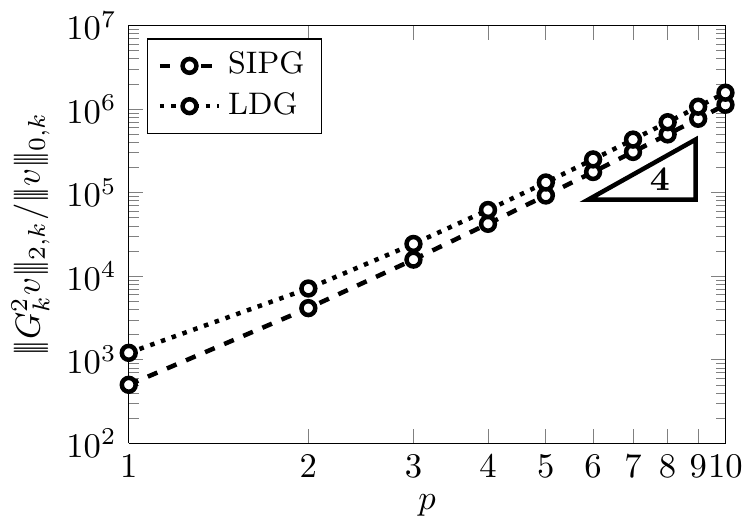}\label{fig:smoothp}}
\hspace{0.2cm}
\subfigure[]
 {\includegraphics[width=0.45\textwidth]{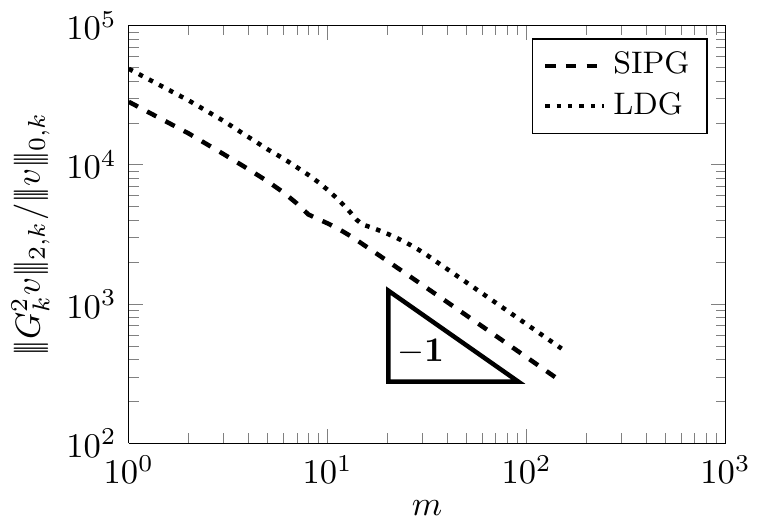}\label{fig:smoothm}}
 \subfigure[]
 {\includegraphics[width=0.45\textwidth]{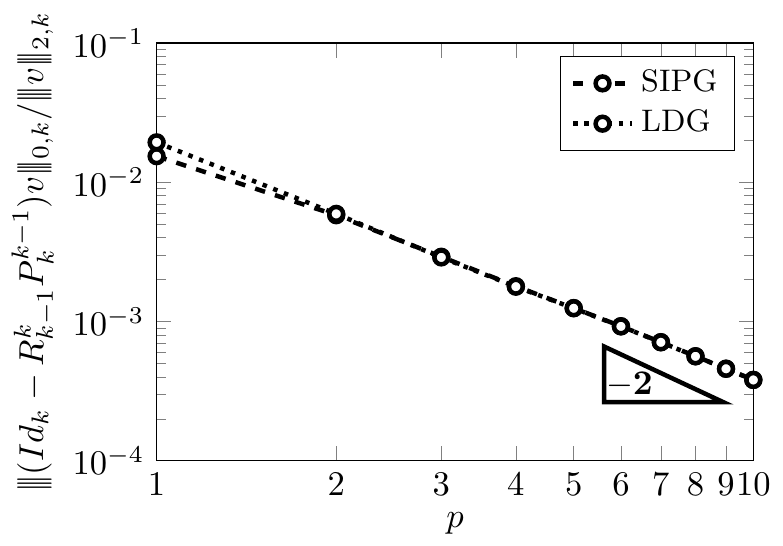}\label{fig:approxp}}
\caption{Estimates of the smoothing property constant as a function of $p$ \subref{fig:smoothp} and $m$ \subref{fig:smoothm} and of the approximation property constant as a function of $p$ \subref{fig:approxp} for the SIPG and LDG methods ($\alpha=10$). Cartesian grid with $h_K = 0.25$.}
\end{figure}

Next, we analyze the convergence factor of the $h$-multigrid iteration of Algorithm~\ref{multilevel}, computing the quantity
\begin{equation*}
\rho = \exp\left(\frac{1}{N}\ln \frac{\|\mathbf{r}_{N}\|_2}{\|\mathbf{r}_{0}\|_2}\right),
\end{equation*}
being $N$ the iteration counts needed to achieve convergence up to a (relative) tolerance of $10^{-8}$ and $\mathbf{r}_{N}$ and $\mathbf{r}_{0}$ being the final and initial residuals, respectively.
For all the test cases, we fix the coarsest level ($\h_1 = 0.25$) and compute
a sequence of nested grids according to the refinement algorithm described in Section~\ref{mulgr}. Table~\ref{tab:hMGvsm} shows the  computed convergence factors obtained with the SIPG and LDG methods ($\alpha=10$, $p=1$) as a function of $m$ and the number of levels, with Cartesian and triangular structured grids, respectively. Here the symbol ``-" means that the convergence has not been reached within a maximum number of 10000 iterations. As predicted in Theorem~\ref{thm:final}, we observe that the convergence factor is independent of the number of levels $k$.
\begin{table}[htb]
\centering
\caption{Convergence factor $\rho$ of $h$-multigrid as a function of $m$ and the number of levels ($\alpha=10$, $p=1$).}
\label{tab:hMGvsm}
\footnotesize
\begin{tabular}{|l|c|c|c|c|c|c|c|c|}
\hline
	&\multicolumn{4}{c|}{SIPG. Cartesian grids.}&\multicolumn{4}{c|}{LDG. Triangular grids.}\\
	\hhline{~--------}
 	&	$k = 2$ 	&	$k = 3$ 	&	$k = 4$ 	&	$k = 5$ 	&	$k = 2$ 	&	$k = 3$ 	&	$k = 4$ 	&	$k = 5$ \\
\hline
$m=1$ 	&	0.8911	&	0.9014	&	0.8955	&	0.8980	&	- 	&	    - 	&	    - 	&	    - 	\\
$m=2$ 	&	0.7958	&	0.8098	&	0.8026	&	0.8050	&	- 	&	    - 	&	    - 	&	    - 	\\
$m=3$ 	&	0.7141  & 0.7311  & 0.7233  & 0.7238	&	- 	&	    - 	&	    - 	&	    - 	\\
$m=4$ 	&	0.6493	&	0.6662	&	0.6587	&	0.6589	&	- 	&	    - 	&	    - 	&	    - 	\\
$m=5$ 	&	0.6043  & 0.6195  & 0.6146  & 0.6129	&	0.9041	&	0.9102	&	0.9097	&	0.9103	\\
$m=6$ 	&	0.5713	&	0.5896	&	0.5873	&	0.5807	&	0.8869	&	0.8927	&	0.8933	&	0.8940	\\
$m=8$ 	&	0.5236	&	0.5442	&	0.5436	&	0.5341	&	0.8542	&	0.8609	&	0.8622	&	0.8628	\\
$m=10$ 	&	0.4847	&	0.4998	&	0.5009	&	0.4906	&	0.8236	&	0.8314	&	0.8330	&	0.8335	\\
$m=12$ 	&	0.4514	&	0.4575	&	0.4605	&	0.4537	&	0.7952	&	0.8042	&	0.8059	&	0.8063	\\
$m=14$ 	&	0.4206	&	0.4159	&	0.4238	&	0.4137	&	0.7689	&	0.7802	&	0.7812	&	0.7814	\\
$m=16$ 	&	0.3916	&	0.3849	&	0.3885	&	0.3790	&	0.7451	&	0.7607	&	0.7596	&	0.7595	\\
$m=18$ 	&	0.3667	&	0.3565	&	0.3560	&	0.3514	&	0.7250	&	0.7453	&	0.7422	&	0.7413	\\
$m=20$ 	&	0.3432	&	0.3348	&	0.3312	&	0.3267	&	0.7087	&	0.7328	&	0.7291	&	0.7266	\\
\hline
\end{tabular}
\end{table}
For the sake of completeness, in Table~\ref{tab:hMGRvsm} we also verify the estimate given in Theorem~\ref{thm:finalR} obtained by considering $A^\textnormal{R}_k$ instead of $A_k$ for the SIPG method on structured triangular grids. As predicted theoretically, the convergence rate increases with the number of levels.

\begin{table}[htb]
\centering
\caption{SIPG method. Convergence factor $\rho$ of $h$-multigrid as a function of $m$ and the number of levels with $A_k^\textnormal{R}$ ($\alpha=10$, $p=1$). Triangular grids.}
\label{tab:hMGRvsm}
\footnotesize
\begin{tabular}{|l|c|c|c|c|c|c|}
\hline
 	&	$k = 2$ 	&	$k = 3$ 	&	$k = 4$ 	&	$k = 5$ 	&	$k = 6$ 	&	$k = 7$ \\
\hline
$m = 1$ & 0.8766 & 0.8978 & 0.9190 & 0.9299 & 0.9352 & 0.9387\\
$m = 2$ & 0.7820 & 0.8296 & 0.8586 & 0.8754 & 0.8840 & 0.8895\\
$m = 3$ & 0.7169 & 0.7793 & 0.8123 & 0.8314 & 0.8418 & 0.8486\\
$m = 4$ & 0.6734 & 0.7422 & 0.7754 & 0.7948 & 0.8060 & 0.8135\\
$m = 5$ & 0.6405 & 0.7132 & 0.7448 & 0.7633 & 0.7749 & 0.7827\\
$m = 6$ & 0.6128 & 0.6882 & 0.7184 & 0.7359 & 0.7473 & 0.7552\\
$m = 8$ & 0.5672 & 0.6437 & 0.6733 & 0.6894 & 0.7005 & 0.7089\\
$m = 10$ & 0.5277 & 0.6046 & 0.6337 & 0.6497 & 0.6611 & 0.6697\\
$m = 12$ & 0.4922 & 0.5687 & 0.5974 & 0.6139 & 0.6257 & 0.6347\\
$m = 14$ & 0.4604 & 0.5358 & 0.5645 & 0.5807 & 0.5930 & 0.6023\\
$m = 16$ & 0.4302 & 0.5046 & 0.5334 & 0.5497 & 0.5620 & 0.5717\\
$m = 18$ & 0.4034 & 0.4763 & 0.5040 & 0.5210 & 0.5333 & 0.5432\\
$m = 20$ & 0.3759 & 0.4508 & 0.4770 & 0.4938 & 0.5062 & 0.5160\\
\hline
\end{tabular}
\end{table}

In Table~\ref{tab:hMGvsp}, we show the iteration counts and convergence factor (between parenthesis) of $h$-multigrid as a function of the polynomial approximation degree $p$ and for different levels $k$, for both SIPG and LDG methods. Here $m = 6$ so that convergence is guaranteed in all the cases. We also compare our results with the iteration counts of the Conjugate Gradient (CG) algorithm. It is clear that the multigrid algorithm converges much faster than CG and, as expected from estimate \eqref{final}, the iteration counts needed to get convergence increases with $p$. However, estimate \eqref{final} seems to be rather pessimistic with respect to numerical simulations.
\begin{table}[htb]
\centering
\caption{Iteration counts and convergence factor (between parenthesis) of  $h$-multigrid as a function of $p$ and the number of levels $k$ ($\alpha=10$, $m=6$).}
\label{tab:hMGvsp}
\footnotesize
\begin{tabular}{|l|c|c|c|c|c|c|}
\hline
	&\multicolumn{3}{c|}{SIPG. Cartesian grids.}&\multicolumn{3}{c|}{LDG. Triangular grids.}\\
	\hhline{~------}
 	&	$k = 2$ 	&	$k = 3$ 	&	$k = 4$ &	$k = 2$ 	&	$k = 3$ 	&	$k = 4$\\
 	\hline
$p = 1$ & 33 (0.57) & 35 (0.59) & 35 (0.59) & 154 (0.89) & 163 (0.89) & 164 (0.89) \\ 
$p = 2$ & 125 (0.86) & 123 (0.86) & 120 (0.86) & 286 (0.94) & 311 (0.94) & 304 (0.94) \\ 
$p = 3$ & 182 (0.90) & 175 (0.90) & 150 (0.88) & 487 (0.96) & 516 (0.96) & 465 (0.96) \\ 
$p = 4$ & 296 (0.94) & 246 (0.93) & 243 (0.93) & 761 (0.98) & 682 (0.97) & 555 (0.97) \\  
$p = 5$ & 407 (0.96) & 354 (0.95) & 362 (0.95) & 838 (0.98) & 644 (0.97) & 576 (0.97) \\ 
$p = 6$ & 553 (0.97) & 483 (0.96) & 489 (0.96) & 856 (0.98) & 659 (0.97) & 741 (0.98) \\ 
\hhline{~------}
&\multicolumn{6}{c|}{CG iteration counts}\\
\hhline{~------}
$ p = 1$ & 65 &  130 &  256 &  286 &  607 & 1218\\
$ p = 2$ & 142 &  281 &  567 &  575 & 1150 & 2343\\
$ p = 3$ & 244 &  499 & 1021 &  914 & 1640 & 3322\\
$ p = 4$ & 400 &  828 & 1687 & 1197 & 2483 & 5053\\
$ p = 5$ & 646 & 1342 & 2721 & 1686 & 3487 & 7090\\
$ p = 6$ & 1130 & 2369 & 4818 & 2286 & 4746 & 9676\\
\hline
\end{tabular}
\end{table}

We next present some numerical results obtained with $p$-multigrid algorithm. We fix the mesh size $\h_k = 0.0625$, for any $k$, while we set $p_{k-1} = p_{k}-1$, with the convention that $p_K = p$. In Table~\ref{tab:pMGvsm} we report the iteration counts and the convergence factor (between parenthesis) as a function of the number of smoothing steps $m$ and the number of levels $k$ for $p = 5$. Since, with $p_{k-1} = p_k-1$, the ratio $p_k/p_{k-1}$ is not constant among the levels, the uniformity with respect to the number of levels is reached asymptotically and is not fully appreciated in Table~\ref{tab:pMGvsm}. As before we also report CG iteration counts: we observe that, even with a relatively small number of pre- and post-smoothing steps, the W-cycle algorithm outperforms CG method.
\begin{table}[htb]
\centering
\caption{Convergence factor $\rho$ of $p$-multigrid as a function of $m$ and the number of levels ($\alpha=10$, $p=5$).}
\label{tab:pMGvsm}
\footnotesize
\begin{tabular}{|l|c|c|c|c|c|c|c|c|}
\hline
	&\multicolumn{3}{c|}{SIPG. Cartesian grids.}&\multicolumn{3}{c|}{LDG. Triangular grids.}\\
	\hhline{~------}
 	&	$k = 2$ 	&	$k = 3$ 	&	$k = 4$ &	$k = 2$ 	&	$k = 3$ 	&	$k = 4$\\
 	\hline
$m = 1$ & 444 (0.96) & 498 (0.96) & 676 (0.97) & - & - & - \\ 
$m = 2$ & 256 (0.93) & 271 (0.93) & 369 (0.95) & 383 (0.95) & 414 (0.96) & 639 (0.97) \\ 
$m = 4$ & 140 (0.88) & 147 (0.88) & 217 (0.92) & 219 (0.92) & 223 (0.92) & 355 (0.95) \\ 
$m = 6$ & 115 (0.85) & 117 (0.85) & 161 (0.89) & 162 (0.89) & 156 (0.89) & 254 (0.93) \\ 
$m = 8$ & 104 (0.84) & 103 (0.84) & 128 (0.87) & 136 (0.87) & 131 (0.87) & 200 (0.91) \\ 
$m = 10$ & 92 (0.82) & 92 (0.82) & 107 (0.84) & 120 (0.86) & 117 (0.85) & 166 (0.89) \\ 
$m = 12$ & 82 (0.80) & 82 (0.80) & 92 (0.82) & 108 (0.84) & 106 (0.84) & 142 (0.88) \\ 
$m = 14$ & 74 (0.78) & 74 (0.78) & 81 (0.80) & 98 (0.83) & 97 (0.83) & 125 (0.86) \\ 
$m = 16$ & 68 (0.76) & 68 (0.76) & 72 (0.77) & 91 (0.82) & 90 (0.81) & 112 (0.85) \\ 
$m = 18$ & 62 (0.74) & 62 (0.74) & 65 (0.75) & 84 (0.80) & 83 (0.80) & 101 (0.83) \\ 
$m = 20$ & 58 (0.73) & 58 (0.73) & 60 (0.73) & 79 (0.79) & 78 (0.79) & 92 (0.82) \\
\hline
&\multicolumn{3}{c|}{CG iteration counts:  $1347$}&\multicolumn{3}{c|}{CG iteration counts:  $3495$}\\
\hline
\end{tabular}
\end{table}
In Table~\ref{tab:pMGvsp},we fix the number of levels and vary the polynomial approximation degree $p=2,3,\ldots,6$, and report the $p$-multigrid iteration counts and the convergence factor (between parenthesis). As before we address the performance of both the SIPG and of the LDG
methods on Cartesian and triangular grids, respectively.
Comparing the iteration counts with the analogous one computed with CG algorithm (last column) we can conclude that even if the iteration counts grows as $p$ increases, the W-cycle algorithm always outperforms CG iterative scheme.
\begin{table}[htb]
\centering
\caption{Iteration counts and convergence factor (between parenthesis) of $p$-multigrid as a function of $p$ and the number of levels ($\alpha=10$, $m=10$).}
\label{tab:pMGvsp}
\footnotesize
\setlength{\tabcolsep}{4.5pt}
\begin{tabular}{|l|c|c|c|c|c|c|c|c|}
\hline
	&\multicolumn{4}{c|}{SIPG. Cartesian grids.}& \multicolumn{4}{c|}{LDG. Triangular grids.}\\
  \hhline{~--------}
 	&	$k = 2$ 	&	$k = 3$ 	&	$k = 4$ & CG &	$k = 2$ 	&	$k = 3$ 	&	$k = 4$ & CG\\
 \hline
$p = 2$ 	&	34 (0.58)	&	-	&	-	&	281 &86 (0.81)	&	-&	-	& 1157\\
$p = 3$ 	&	78 (0.79)	&	76 (0.78)	&	-	& 499 &	104	(0.84)&	107 (0.84)&	-	& 1639\\
$p = 4$ 	&	71 (0.77)	&	74 (0.78)	&	75 (0.78)& 828	&	115	(0.85)&	119	(0.86)&	143 (0.88)	& 2491\\
$p = 5$ 	&	92 (0.82)	&	92 (0.82)&	107	(0.84)&	1347 & 120	(0.86)&	117	(0.85)&	166	(0.89)	& 3495\\
$p = 6$ 	&	113	(0.85) &	113 (0.85)	&	109	(0.84) & 2370 &	145	(0.88)&	143	(0.88)&	158	(0.89)& 4737\\
\hline
\end{tabular}
\end{table}
\section{Conclusions}
\label{conclusion}

We have analyzed a W-cycle $hp$-multigrid scheme for high order DG discretizations of elliptic problems. We have shown uniform convergence with respect to $h$, provided the number of pre- and post-smoothing steps is sufficiently large, and we have tracked the dependence of the convergence factor of the method on the polynomial approximation degree $p$. Besides the traditional approach, where the coarse matrices are built on each level \cite{GopKan03,BrenZhao,BrennCuiSu,BrenCuGuSu}, we have also considered the case of inherited bilinear forms, showing that the rate of convergence cannot be uniform with respect to the number of levels. Finally, the theoretical results obtained in this paper pave the way for future developments in obtaining uniform $p$-multigrid methods by introducing more sophisticated smoothing schemes. Such an issue will be object of future research.

\FloatBarrier
\appendix
\section{Proof of Theorem~\ref{thm:DGerr}} \label{appendixA}
Before proving Theorem~\ref{thm:DGerr}, we recall the following $hp$-approximation result and report its proof for the sake of completeness.
\begin{lemma}
For any $v\in H^{s+1}(\mathcal{T}_K)$, there exists $\Pi_{h_K}^{p_K}v \in V_K$, $p_K=1,2,\ldots,$ such that
\begin{equation}
\label{interpDG}
\|v-\Pi_{h_K}^{p_K}v\|_{\textnormal{DG},K}\lesssim\frac{\h_K^{\min{(p_K,s)}}}{p_K^{s-1/2}}\|v\|_{H^{s+1}(\mathcal{T}_K)}.
\end{equation}
\end{lemma}
\begin{proof}
For any $v\in H^{s+1}(\mathcal{T}_K)$, let $\Pi_{h_K}^{p_K}v \in V_K$ be defined as $\Pi_{h_K}^{p_K}v|_T=\pi_{h_K}^{p_K}(v|_T)$, for any $T \in \mathcal{T}_K$, $\pi_{h_K}^{p_K}(\cdot)$ being the Babu{\v{s}}ka-Suri interpolant \cite{BabSur87}. From \cite[Lemma~4.5]{BabSur87} we have that
\begin{equation}
\label{interp1}
\begin{aligned}
\|u-\Pi_{h_K}^{p_K}u\|_{H^q(T)}\lesssim \frac{h_T^{\min(p_T+1,s+1)-q}}{p_T^{s+1-q}}\|u\|_{H^{s+1}(T)},
&& 0\leq q\leq t
\end{aligned}
\end{equation}
for any $T \in \mathcal{T}_K$. Moreover, as suggested in \cite{HouSchSul}, by a multiplicative trace inequality and \eqref{interp1}, we also have
\begin{align}
\label{interp2}
\|u-\Pi_{h_K}^{p_K}u\|_{L^2(\partial T)}^2&\lesssim \|u-\Pi_{h_K}^{p_K}u\|_{L^2(T)}\|\nabla(u-\Pi_{h_K}^{p_K}u)\|_{L^2(T)}+h_T^{-1}\|u-\Pi_{h_K}^{p_K}u\|_{L^2(T)}^2\notag\\ 
&\lesssim\frac{h_T^{2\min(p_T,s)+1}}{p_T^{2s+1}}\|u\|_{H^{s+1}(T)}^2,
\end{align}
for any $T \in \mathcal{T}_K$.
The thesis follows applying \eqref{interp1} and \eqref{interp2} to the definition \eqref{DGnorm} of the DG norm.
\end{proof}

We are now ready to prove Theorem~\ref{thm:DGerr}.
\begin{proof}{[Proof of Theorem~\ref{thm:DGerr}.]}
The proof follows the lines given in \cite{PerSchot}; for the sake of completeness we sketch it.
It can be shown that formulation \eqref{DG} is not strongly consistent, {\it i.e.}, 
\begin{equation}
\label{Galortho}
\mathcal{A}_K(u-u_K,v) = R(u,v)\quad\forall v\in V(\h_K),
\end{equation}
where the residual $R(\cdot,\cdot):V(\h_K)\times V(\h_K)\rightarrow \mathbb{R}$ is defined as
\begin{equation}
R(u,v) = \sum_{T\in\mathcal{T}_K}\int_T\nabla u\cdot\left[\mathcal{R}_K(\llbracket v\rrbracket)+\mathcal{L}_K(\boldsymbol{\beta}\cdot\llbracket z\rrbracket)\right]\ dx
+\sum_{F\in\mathcal{F}_K}\int_F\llaverage\nabla u\rraverage\cdot\llbracket v\rrbracket\ ds.
\end{equation}
As shown in \cite{PerSchot}, the residual $R(\cdot,\cdot)$ can be bounded by
\begin{equation}
\label{residualest}
|R(u,v)|\lesssim \frac{h_K^{\min(p_K+1,s)}}{p_K^{s}}\|\nabla u\|_{H^{s}(\mathcal{T}_K)}\|v\|_{\textnormal{DG},K}\qquad\forall v\in V(\h_K).
\end{equation}
From the continuity and coercivity bounds \eqref{cont} and \eqref{coerc}, and Strang's lemma, we have
\begin{equation}
\label{strang}
\|u-u_K\|_{DG,K} \lesssim \inf_{v\in V_K} \|u-v\|_{DG,K}+\sup_{w\in V_K} \frac{|R(u,w)|}{\|w\|_{DG,K}}.
\end{equation}
Inequality \eqref{DGerror1} follows by choosing $v = \Pi_{h_K}^{p_K}u \in V_K$, and substituting \eqref{interpDG} and the residual estimate \eqref{residualest} in \eqref{strang}.
With regards to estimate \eqref{L2error1}, we proceed by a standard duality argument: let $w\in  H^2(\Omega) \cap H_0^1(\Omega)$ be the solution of the problem
\begin{equation*}
\int_\Omega \nabla w \cdot \nabla v\ dx = \int_\Omega (u-u_K)v\ dx\qquad \forall v\in H_0^1(\Omega).
\label{adjoint}
\end{equation*}
We recall that, thanks to the regularity assumption \eqref{elliptic}, it holds that
\begin{equation*}
\left\|w\right\|_{H^{2}(\Omega)}\lesssim \left\|u-u_K\right\|_{{L^2}(\Omega)}.
\label{ellipticw}
\end{equation*}
According to \eqref{Galortho}, it is immediate to obtain
\begin{equation}
\|u-u_K\|_{L^2(\Omega)}^2 = \mathcal{A}_K(w,u-u_K)-R(w,u-u_K),
\end{equation}
and 
\begin{equation*}
\mathcal{A}_K(w_I,u-u_K) = R(u,w_I) = -R(u,w-w_I),\\
\end{equation*}
with $w_I\in V_K$. Hence,
$$\|u-u_K\|_{L^2(\Omega)}^2 = \mathcal{A}_K(w-w_I,u-u_K)-R(w,u-u_K)-R(u,w-w_I).$$
Applying continuity \eqref{cont} and the residual estimate \eqref{residualest}, we obtain
\begin{multline*}
\|u-u_K\|_{L^2(\Omega)}^2
\lesssim \|w-w_I\|_{\textnormal{DG},K}\|u-u_K\|_{\textnormal{DG},K}+\frac{\h_K}{p_K}\|\nabla w\|_{H^1(\Omega)}\|u-u_K\|_{\textnormal{DG},K}\\
+\frac{\h_K^{\min(p_K+1,s)}}{p_K^s}\|\nabla u\|_{H^s(\mathcal{T}_K)}\|w-w_I\|_{\textnormal{DG},K}.
\end{multline*}
If we choose $w_I=\Pi_{h_K}^{p_K}w$, it holds that
\begin{equation}
\|w-w_I\|_{\textnormal{DG},K}\lesssim\frac{\h_K}{p_K^{1/2}}\|w\|_{H^2(\Omega)}\lesssim\frac{\h_K}{p_K^{1/2}}\|u-u_K\|_{L^2(\Omega)},
\end{equation}
which together with \eqref{DGerror1} and $\|\nabla w\|_{H^1(\Omega)}\lesssim \|u-u_K\|_{L^2(\Omega)}$ gives the thesis.
\end{proof}

\bibliographystyle{siam}
\bibliography{biblio}

\begin{thebibliography}{10}

\bibitem{Adams}
{\sc R.~A. Adams}, {\em Sobolev spaces}, Academic Press [A subsidiary of
  Harcourt Brace Jovanovich, Publishers], New York-London, 1975.
\newblock Pure and Applied Mathematics, Vol. 65.

\bibitem{AntoAyu07}
{\sc P.~F. Antonietti and B.~Ayuso}, {\em Schwarz domain decomposition
  preconditioners for discontinuous {G}alerkin approximations of elliptic
  problems: non-overlapping case}, M2AN Math. Model. Numer. Anal., 41 (2007),
  pp.~21--54.

\bibitem{AntoAyu08}
\leavevmode\vrule height 2pt depth -1.6pt width 23pt, {\em Multiplicative
  {S}chwarz methods for discontinuous {G}alerkin approximations of elliptic
  problems}, M2AN Math. Model. Numer. Anal., 42 (2008), pp.~443--469.

\bibitem{AntoAyu09}
\leavevmode\vrule height 2pt depth -1.6pt width 23pt, {\em Two-level {S}chwarz
  preconditioners for super penalty discontinuous {G}alerkin methods}, Commun.
  Comput. Phys., 5 (2009), pp.~398--412.

\bibitem{AntAyuBerPen12}
{\sc P.~F. Antonietti, B.~Ayuso, S.~Bertoluzza, and M.~Penacchio}, {\em
  {Substructuring Preconditioners for an h-p Nitsche-type method}}, Tech.
  Report IMATI-PV\ n. 17PV12/16/0, 2012.
\newblock Submitted.

\bibitem{AntAyuBrenSung12}
{\sc P.~F. Antonietti, B.~Ayuso, S.~C. Brenner, and L.-Y. Sung}, {\em Schwarz
  methods for a preconditioned {WOPSIP} method for elliptic problems}, Comput.
  Meth. in Appl. Math., 12 (2012), pp.~241--272.

\bibitem{AntGiaHou13}
{\sc P.~F. Antonietti, S.~Giani, and P.~Houston}, {\em {Domain decomposition
  preconditioners for Discontinuous Galerkin methods for elliptic problems on
  complicated domains}}, MOX report 15/2013,  (2013).
\newblock Submitted.

\bibitem{AntHou}
{\sc P.~F. Antonietti and P.~Houston}, {\em A class of domain decomposition
  preconditioners for {$hp$}-discontinuous {G}alerkin finite element methods},
  J. Sci. Comput., 46 (2011), pp.~124--149.

\bibitem{Arn82}
{\sc D.~N. Arnold}, {\em An interior penalty finite element method with
  discontinuous elements}, SIAM J. Numer. Anal., 19 (1982), pp.~742--760.

\bibitem{Arn}
{\sc D.~N. Arnold, F.~Brezzi, B.~Cockburn, and L.~D. Marini}, {\em Unified
  analysis of discontinuous {G}alerkin methods for elliptic problems}, SIAM J.
  Numer. Anal., 39 (2001/02), pp.~1749--1779.

\bibitem{BabSur87}
{\sc I.~Babu{\v{s}}ka and M.~Suri}, {\em The {$h$}-{$p$} version of the finite
  element method with quasi-uniform meshes}, RAIRO Mod\'el. Math. Anal.
  Num\'er., 21 (1987), pp.~199--238.

\bibitem{BarkBrennSung11}
{\sc A.~T. Barker, S.~C. Brenner, and L.-Y. Sung}, {\em Overlapping {S}chwarz
  domain decomposition preconditioners for the local discontinuous {G}alerkin
  method for elliptic problems}, J. Numer. Math., 19 (2011), pp.~165--187.

\bibitem{BasGhid}
{\sc F.~Bassi, A.~Ghidoni, S.~Rebay, and P.~Tesini}, {\em {High-order accurate
  p-multigrid discontinuous Galerkin solution of the Euler equations}},
  Internat. J. Numer. Methods Fluids, 60 (2009), pp.~847--865.

\bibitem{Basetal}
{\sc F.~Bassi, S.~Rebay, G.~Mariotti, S.~Pedinotti, and M.~Savini}, {\em A
  high-order accurate discontinuous finite element method for inviscid and
  viscous turbomachinery flows}, in Proceedings of the 2nd European Conference
  on Turbomachinery Fluid Dynamics and Thermodynamics.

\bibitem{bramble}
{\sc J.H. Bramble}, {\em Multigrid Methods}, no.~294 in Pitman Research Notes
  in Mathematics Series, Longman Scientific \& Technical, 1993.

\bibitem{Bren99}
{\sc S.~C. Brenner}, {\em Convergence of nonconforming multigrid methods
  without full elliptic regularity}, Math. Comp., 68 (1999), pp.~25--53.

\bibitem{BrenCuGuSu}
{\sc S.~C. Brenner, J.~Cui, T.~Gudi, and L.-Y. Sung}, {\em Multigrid algorithms
  for symmetric discontinuous {G}alerkin methods on graded meshes}, Numer.
  Math., 119 (2011), pp.~21--47.

\bibitem{BrennCuiSu}
{\sc S.~C. Brenner, J.~Cui, and L.-Y. Sung}, {\em Multigrid methods for the
  symmetric interior penalty method on graded meshes}, Numer. Linear Algebra
  Appl., 16 (2009), pp.~481--501.

\bibitem{BrennOwe}
{\sc S.~C. Brenner and L.~Owens}, {\em A {$W$}-cycle algorithm for a weakly
  over-penalized interior penalty method}, Comput. Methods Appl. Mech. Engrg.,
  196 (2007), pp.~3823--3832.

\bibitem{BrennParSu13}
{\sc S.~C. Brenner, E.-H. Park, and L.-Y. Sung}, {\em A balancing domain
  decomposition by constraints preconditioner for a weakly over-penalized
  symmetric interior penalty method}, Numerical Linear Algebra with
  Applications, 20 (2013), pp.~472--491.

\bibitem{BrenScott}
{\sc S.~C. Brenner and L.~R. Scott}, {\em The mathematical theory of finite
  element methods}, vol.~15 of Texts in Applied Mathematics, Springer, New
  York, third~ed., 2008.

\bibitem{BrenZhao}
{\sc S.~C. Brenner and J.~Zhao}, {\em Convergence of multigrid algorithms for
  interior penalty methods}, Appl. Numer. Anal. Comput. Math., 2 (2005),
  pp.~3--18.

\bibitem{BreMan}
{\sc F.~Brezzi, G.~Manzini, D.~Marini, P.~Pietra, and A.~Russo}, {\em
  Discontinuous finite elements for diffusion problems}, in Atti convegno in
  onore di F. Brioschi, Istituto Lombardo, Accademia di Scienze e Lettere
  (1999).

\bibitem{BrCampCanDah13}
{\sc K.~{Brix}, M.~{Campos Pinto}, C.~{Canuto}, and W.~{Dahmen}}, {\em
  {Multilevel Preconditioning of Discontinuous-Galerkin Spectral Element
  Methods, Part I: Geometrically Conforming Meshes}}, IGPM preprint \#355.
\newblock Submitted.

\bibitem{BriCamPinDahmMass09}
{\sc K.~Brix, M.~Campos~Pinto, W.~Dahmen, and R.~Massjung}, {\em {Multilevel
  preconditioners for the interior penalty discontinuous Galerkin method {II} -
  Quantitative studies}}, Commun. Comput. Phys., 5 (2009), pp.~296--325.

\bibitem{BriCamPinDahm08}
{\sc K.~Brix, M.C. Pinto, and W.~Dahmen}, {\em {A multilevel preconditioner for
  the interior penalty discontinuous Galerkin method}}, SIAM J. Numer. Anal.,
  46 (2008), pp.~2742--2768.

\bibitem{CanPavPie12}
{\sc C.~Canuto, L.~F. Pavarino, and A.B. Pieri}, {\em {BDDC} preconditioners
  for continuous and discontinuous {G}alerkin methods using spectral/hp
  elements with variable polynomial degree},  (2012).
\newblock Submitted.

\bibitem{CoShu}
{\sc B.~Cockburn and C.-W. Shu}, {\em The local discontinuous {G}alerkin method
  for time-dependent convection-diffusion systems}, SIAM J. Numer. Anal., 35
  (1998), pp.~2440--2463 (electronic).

\bibitem{DioDarm12}
{\sc L.~T. Diosady and D.~L. Darmofal}, {\em A unified analysis of balancing
  domain decomposition by constraints for discontinuous {G}alerkin
  discretizations}, SIAM J. Numer. Anal., 50 (2012), pp.~1695--1712.

\bibitem{DryGalSark07}
{\sc M.~Dryja, J.~Galvis, and M.~Sarkis}, {\em B{DDC} methods for discontinuous
  {G}alerkin discretization of elliptic problems}, J. Complexity, 23 (2007),
  pp.~715--739.

\bibitem{DrySark06}
{\sc M.~Dryja and M.~Sarkis}, {\em A Neumann-Neumann Method for DG
  Discretization of Elliptic Problems}, Informes de matem{\'a}tica, Inst. de
  Matem{\'a}tica Pura e Aplicada, 2006.

\bibitem{DrySark10}
\leavevmode\vrule height 2pt depth -1.6pt width 23pt, {\em Additive average
  {S}chwarz methods for discretization of elliptic problems with highly
  discontinuous coefficients}, Comput. Methods Appl. Math., 10 (2010),
  pp.~164--176.

\bibitem{FengKarak01}
{\sc X.~Feng and O.~A. Karakashian}, {\em Two-level additive {S}chwarz methods
  for a discontinuous {G}alerkin approximation of second order elliptic
  problems}, SIAM J. Numer. Anal., 39 (2001), pp.~1343--1365 (electronic).

\bibitem{FengKarak05}
\leavevmode\vrule height 2pt depth -1.6pt width 23pt, {\em Two-level
  non-overlapping {S}chwarz preconditioners for a discontinuous {G}alerkin
  approximation of the biharmonic equation}, J. Sci. Comput., 22/23 (2005),
  pp.~289--314.

\bibitem{FidkOli}
{\sc K.~J. Fidkowski, T.~A. Oliver, J.~Lu, and D.~L. Darmofal}, {\em
  {p-Multigrid solution of high-order discontinuous Galerkin discretizations of
  the compressible Navier-Stokes equations}}, J. Comput. Phys., 207 (2005),
  pp.~92--113.

\bibitem{GeorgSul05}
{\sc E.~H. Georgoulis and E.~S{\"u}li}, {\em Optimal error estimates for the
  {$hp$}-version interior penalty discontinuous {G}alerkin finite element
  method}, IMA J. Numer. Anal., 25 (2005), pp.~205--220.

\bibitem{GopKan03}
{\sc J.~Gopalakrishnan and G.~Kanschat}, {\em A multilevel discontinuous
  {G}alerkin method}, Numer. Math., 95 (2003), pp.~527--550.

\bibitem{Hack}
{\sc W.~Hackbusch}, {\em {Multi-grid methods and applications}}, vol.~4 of
  Springer series in computational mathematics, Springer, Berlin, 1985.

\bibitem{HeinPie}
{\sc B.~Heinrich and K.~Pietsch}, {\em Nitsche type mortaring for some elliptic
  problem with corner singularities}, Computing, 68 (2002), pp.~217--238.

\bibitem{HemHoffRaa03}
{\sc P.~W. Hemker, W.~Hoffmann, and M.~H. van Raalte}, {\em Two-level {F}ourier
  analysis of a multigrid approach for discontinuous {G}alerkin
  discretization}, SIAM J. Sci. Comput., 25 (2003), pp.~1018--1041
  (electronic).

\bibitem{HemHoffRaa04}
\leavevmode\vrule height 2pt depth -1.6pt width 23pt, {\em Fourier two-level
  analysis for discontinuous {G}alerkin discretization with linear elements},
  Numer. Linear Algebra Appl., 11 (2004), pp.~473--491.

\bibitem{HouSchSul}
{\sc P.~Houston, C.~Schwab, and E.~S{\"u}li}, {\em Discontinuous {$hp$}-finite
  element methods for advection-diffusion-reaction problems}, SIAM J. Numer.
  Anal., 39 (2002), pp.~2133--2163.

\bibitem{LassTos03}
{\sc C.~Lasser and A.~Toselli}, {\em An overlapping domain decomposition
  preconditioner for a class of discontinuous {G}alerkin approximations of
  advection-diffusion problems}, Math. Comp., 72 (2003), pp.~1215--1238
  (electronic).

\bibitem{Luo}
{\sc H.~Luo, J.~D. Baum, and R.~L\"{o}hner}, {\em {A p-multigrid discontinuous
  Galerkin method for the Euler equations on unstructured grids}}, J. Comput.
  Phys., 211 (2006), pp.~767--783.

\bibitem{MascaHelen}
{\sc B.~S. Mascarenhas, B.~T. Helenbrook, and H.~L. Atkins}, {\em Coupling
  {$p$}-multigrid to geometric multigrid for discontinuous {G}alerkin
  formulations of the convection-diffusion equation}, J. Comput. Phys., 229
  (2010), pp.~3664--3674.

\bibitem{Nastase}
{\sc C.~R. Nastase and D.~J. Mavriplis}, {\em {High-order discontinuous
  Galerkin methods using an hp-multigrid approach}}, J. Comput. Phys., 213
  (2006), pp.~330--357.

\bibitem{PerSchot}
{\sc I.~Perugia and D.~Sch{\"o}tzau}, {\em An {$hp$}-analysis of the local
  discontinuous {G}alerkin method for diffusion problems}, J. Sci. Comput., 17
  (2002), pp.~561--571.

\bibitem{QuaVal}
{\sc A.~Quarteroni and A.~Valli}, {\em Numerical approximation of partial
  differential equations}, vol.~23 of Springer Series in Computational
  Mathematics, Springer-Verlag, Berlin, 1994.

\bibitem{SchoberlLehrenfeld_2013}
{\sc J.~Sch{\"o}berl and C.~Lehrenfeld}, {\em Domain decomposition
  preconditioning for high order hybrid discontinuous galerkin methods on
  tetrahedral meshes}, Lecture Notes in Applied and Computational Mechanics, 66
  (2013), pp.~27--56.
\newblock cited By (since 1996)0.

\bibitem{ShotzSchwTos02}
{\sc D.~Sch{\"o}tzau, C.~Schwab, and A.~Toselli}, {\em Mixed {$hp$}-{DGFEM} for
  incompressible flows}, SIAM J. Numer. Anal., 40 (2002), pp.~2171--2194
  (electronic) (2003).

\bibitem{Shahbazi}
{\sc K.~Shahbazi, D.~J. Mavriplis, and N.~K. Burgess}, {\em {Multigrid
  algorithms for high-order discontinuous Galerkin discretizations of the
  compressible Navier-Stokes equations}}, J. Comput. Phys., 228 (2009),
  pp.~7917--7940.

\bibitem{StaWihl}
{\sc B.~Stamm and T.~P. Wihler}, {\em {$hp$}-optimal discontinuous {G}alerkin
  methods for linear elliptic problems}, Math. Comp., 79 (2010),
  pp.~2117--2133.

\bibitem{RaaHem05}
{\sc M.~H. van Raalte and P.~W. Hemker}, {\em Two-level multigrid analysis for
  the convection-diffusion equation discretized by a discontinuous {G}alerkin
  method}, Numer. Linear Algebra Appl., 12 (2005), pp.~563--584.

\end{thebibliography}

\end{document}